%% file: Applications_of_the_Likelihood_Theory_in_Finance_arXiv.tex
\documentclass[twoside]{article}
\input{template007}
\title{Applications of the Likelihood Theory in Finance: Modelling and Pricing}
\author{Arnold Janssen, Martin Tietje}

\usepackage{color}
\newcommand{\ind}{1\!\! 1}

\begin{document}
\maketitle\thispagestyle{empty}


\begin{summary}
This paper discusses the connection between mathematical finance and statistical modelling which turns out to be more than a formal mathematical correspondence. 
We like to figure out how common results and notions in statistics and their meaning can be translated to the world of mathema-tical finance and vice versa. 
A lot of similarities can be expressed in terms of LeCam's theory for statistical experiments which is the theory of the behaviour of likelihood processes.\\
For positive prices the arbitrage free financial assets fit into filtered experiments. It is shown that they are given by filtered likelihood ratio processes.
From the statistical point of view, martingale measures, completeness and pricing formulas are revisited. The pricing formulas for various options are connected with the power functions of tests. 
For instance the Black-Scholes price of a European option has an interpretation as Bayes risk of a Neyman Pearson test. Under contiguity the convergence of financial experiments and option prices are obtained. 
In particular, the approximation of It\^{o} type price processes by discrete models and the convergence of associated option prices is studied. 
The result relies on the central limit theorem for statistical experiments, which is well known in statistics in connection with local asymptotic normal (LAN) families. 
As application certain continuous time option prices can be approximated by related discrete time pricing formulas.\\
\\
\textbf{R\'{e}sum\'{e}}\\
Dans cet travail on discute une connexion entre les math\'{e}matiques financi\`{e}res et la mod\'{e}lisation 
statistique. On trouve qu'il y a plus qu'une correspondance formelle. Nous voulons d\'{e}couvrir comment on 
peut traduire des r\'{e}sultats usit\'{e}s statistiques et leur signification par des expressions des 
math\'{e}matiques financi\`{e}res et vice versa. Beaucoup des ressemblances peuvent \^{e}tre exprim\'{e}s 
par la th\'{e}orie de LeCam sur les exp\'{e}riences statistiques o\`{u}, autrement dit, par la th\'{e}orie 
de comportment des processus de likelihood ratios.\\
Un actif financier positif avec l'hypoth\`{e}se d'absence d'opportunit\'{e}s d'arbitrage permet une 
repr\'{e}-sentation par une exp\'{e}rience filtr\'{e}e ou, pour le dire dans une fa\c{c}on plus 
pr\'{e}cise, par un processus de likelihood ratios filtr\'{e}. En utilisant cette observation nous revisitons
des probabilit\'{e}s martingales, l'hypoth\`{e}se de compl\'{e}tude des march\'{e}s et des formules 
d'\'{e}valuation des produits d\'{e}eriv\'{e}s d'un point de vue statistique.
\end{summary}

\renewcommand{\thefootnote}{}
\footnotetext{\hspace*{-.51cm}AMS 1991 subject classification: Primary: 91B02; Secondary: 62B15, 91B24\\ %
}

\section{Introduction}\label{section1}
This paper studies financial models in the light of statistical experiments known from LeCam's theory. 
There are already known a lot of similarities between mathematical finance and statistical models which are summarized below. 
We will see that various aspects of financial models can be treated by well established statistical methods which were originally developed for modelling statistical experiments. 
It is our aim to review and to refer to this type of parallel work.
\begin{itemize}
 \item For researchers in mathematical finance we present likelihood methods of statistics which allow at least partially another view and new insight in the structure of financial models.
\item The statisticians will see which kind of likelihood methods can be used for financial models and where parallel working was done.
\end{itemize}
We want to stress at the very beginning that this paper does not deal with data analysis in finance. It is concerned with modelling. To bring both areas together the paper has the following structure.
As service for all readers major results about statistical experiments are recalled and outlined with complete proofs for those parts of LeCam's theory that are really needed below. 
It is not the intention to present the most general and abstract known form of that work. On the other hand we restrict ourselves to nonnegative price processes in finance. We know of course that more general models exist.
In our context classical questions about completeness, pricing formulas, prices and the convergence of price processes in finance are addressed and linked to statistics.\\
We summarize some applications of methods from statistics in finance which serve as our motivation.
The Neyman-Pearson lemma may be used to solve optimization problems which is done in F\"ollmer and Leukert \cite{Foellmer and Leukert:99} in the case of quantile hedging, 
in F\"ollmer and Leukert \cite{Foellmer and Leukert:00} for minimizing shortfall risk when hedging and for a special case of a simple risk measure in Schied \cite{Schied:04} 
where the risk of the terminal liability of an issued claim is minimized under some constraints on the issuer's capital. 
Considering risk measures which rely on likelihood ratios as in Schied \cite{Schied:04} may lead to the theory of maximin testing treated, for example, by Cvitanic and Karatzas \cite{Cvitanic and Karatzas:01} or, using a different approach, by 
Rudloff and Karatzas \cite{Rudloff and Karatzas:10}. 
In the same spirit Schied \cite{Schied:05} reduces maximization of utility where a robust utility functional is defined by a set of probability measures to the problem of finding a ``least favorable'' probability measures 
by applying a result of Huber and Strassen \cite{Huber and Strassen:73} from test theory. 
The notion of contiguity is used in finance by Kabanov and Kramkov \cite{Kabanov and Kramkov:94} and Hubalek and Schachermayer \cite{Hubalek and Schachermayer:98}, which we will comment on later on. 
Shiryaev \cite{Shiryaev:99} also uses contiguity as well as techniques from LeCam's theory, for instance LeCam's third lemma. 
Gushchin and Mordecki \cite{Gushchin and Mordecki:02} apply the theory of binary experiments to obtain extremal measures which give upper and lower bounds for the range of option prices in semimartingale models. 
For the theory of statistical experiments we refer to LeCam \cite{LeCam:86}, Strasser \cite{Strasser:85a}, Torgersen \cite{Torgersen:91} and Shiryaev and Spokoiny \cite{Shiryaev and Spokoiny:00}. 
The starting point in section $2$ is the following basic observation.
\begin{itemize}
 \item Each positive financial price process which admits at least one martingale measure is a filtered likelihood ratio process.
\end{itemize}
This fact is a useful link to statistical modelling.
\begin{itemize}
 \item In that case there is a statistical model behind the financial price process. It is called a financial statistical experiment.
\end{itemize}
We explain what are the statistical counterparts and dual objects of certain options, option prices and further fundamentals of finance.
Note that the theory of statistical experiments is just the theory of likelihood ratio processes which is well developed in the literature. The correspondence is one to one if and only if the martingale measure is unique.
Otherwise the martingale measure serves as a nuisance parameter. We do not wonder that several authors used statistical likelihood methods in finance.
As a first application, in section $3$ we show that certain option prices can be interpreted in terms of power functions of tests and thus can be regarded as Bayes risk in various cases. 
Here Neyman Pearson tests show up for European put and call options.\\
Section $4$ deals with sequential aspects of financial models and filtered likelihoods. Example \ref{ito} explains that standard It\^{o} type price processes have certain regression models as statistical counterparts.
In addition it is shown that completeness of financial markets can be interpreted as statistical completeness for the corresponding financial experiments.\\
Section $5$ first recalls asymptotic results for statistical experiments which are now applied to financial experiments. 
Roughly speaking the convergence of financial experiments implies the convergence of the power of tests and also the convergence of Bayes risks. We show how to apply these results to the convergence of options prices in finance.
The main tool is contiguity and LeCam's third lemma.\\
Section $6$ deals with the asymptotics of discrete financial markets and their option prices. 
Notice that LeCam's work offers a ``central limit theorem'' for statistical experiments which is known as the famous local asymptotic normality (LAN). 
As nice application of LAN it is shown that various discrete time financial models converge to Black-Scholes type models. As a consequence the convergence of discrete time option prices to Black-Scholes prices can be discussed.\\
Throughout the paper the results are illustrated by simple models, for instance the Cox-Ross-Rubinstein model or Black-Scholes type models.

\section{Price- and Likelihood processes}\label{section2}

In this section we show that price processes can be regarded as likelihood ratio processes of a filtered experiment.
Recall that a filtered experiment is given by a statistical experiment $E=(\Omega, \mathcal{F}, \{P_{\theta} : \theta \in \Theta \})$, including the set of probability spaces $(\Omega, \mathcal{F}, P_{\theta})$, 
with the family of measures $\{P_{\theta} : \theta \in \Theta \}$ being parameterized by a parameter $\theta$, and a filtration $(\mathcal{F}_t)_{t \ge 0}$, see for example Shiryaev and Spokoiny \cite{Shiryaev and Spokoiny:00} (p.1). 
For early works on filtered experiments we refer the reader to Jacod \cite{Jacod:89} and Strasser \cite{Strasser:87} or for a more recent work to Norberg \cite{Norberg:04}, where filtered experiments are compared.\\
A main issue in finance is the existence of a martingale measure which is, at least in the discrete case, equivalent to the ``no arbitrage'' condition if the time horizon is finite, see Shiryaev \cite{Shiryaev:99} (p. 656). 
The continuous time case is discussed in Delbaen and Schachermayer \cite{Delbaen and Schachermayer:94}.\\
Below let us fix a time interval $[0,T]$ with $T<\infty$, a trading period $I \subset [0,T]$ with $\{0,T\} \subset I$, a filtered probability space $(\Omega, \mathcal{F} ,P)$ with filtration $(\mathcal{F}_t)_{t \in I}$ and 
$\mathcal{F}=\sigma(\mathcal{F}_t:t \in I)$ as well as $d$ adapted positive discounted price processes $(X_t^i)_{t \in I}$, $1 \le i \le d$. 
Suppose that $\mathcal{F}_0=\sigma(\mathcal{N})$ with $\mathcal{N}=\{N:P(N)=0 \textnormal{  or  } P(N)=1\}$.
We refer to $P$ as the real world measure.
Our first result characterizes underlying martingale measures of price processes in terms of statistical experiments. The equivalence below is very useful.

\begin{theorem}\label{darstellung}
Let $Q$ be a probability measure equivalent to $P$. The following assertions are equivalent:
\begin{enumerate}
   \item[(1)] There are probability measures $Q_1, ..., Q_d$ on $(\Omega, \mathcal{F})$ satisfying
 \begin{equation}\label{likelihood}
    \frac{dQ_{i|\mathcal{F}_t}}{dQ_{|\mathcal{F}_t}} = \frac{X^i_t}{X^i_0}, \ \ \ t \in I, 
 \end{equation}
and $Q_i \ll Q$ for all $1 \le i \le d$.
   \item[(2)] $Q$ is a martingale measure, i.e. $(X_t^i)_{t \in I}$ is a $Q$-martingale for each $1 \le i \le d$.
\end{enumerate}
\end{theorem}

\begin{proof}
$\textit{(1)}$ $\Rightarrow$ $\textit{(2)}$: 
It is well known that the left hand side of \eqref{likelihood} is a $Q$-martingale as a density process which implies the result.\\
For the implication $\textit{(2)}$ $\Rightarrow$ $\textit{(1)}$
we may define probability measures \ $Q_i \ll Q$ by
\begin{eqnarray*}
\frac{dQ_i}{dQ} := \frac{X^i_T}{X^i_0} \quad \textnormal{since }~ E_Q\left(\frac{X_T^i}{X_0^i}\right)=\frac{X_0^i}{X_0^i}=1.
\end{eqnarray*}
This gives 
\begin{eqnarray*}
	\frac{X^i_t}{X^i_0}
	= E_Q\left[\frac{X_T^i}{X_0^i} \Bigg|\mathcal{F}_t \right]=E_Q\left[\frac{dQ_i}{dQ} \Bigg|\mathcal{F}_t \right]
	= \frac{dQ_{i|\mathcal{F}_t}}{dQ_{|\mathcal{F}_t}}
\end{eqnarray*}
for all $1 \le i \le d$ and $t \in I$, which completes the proof.
\end{proof}

\begin{remark}
 For $I=[0,\infty)$ Theorem \ref{darstellung} remains true only if we enforce further conditions. The implication $\textit{(1)}$ $\Rightarrow$ $\textit{(2)}$ carries over to the case $I=[0,\infty)$ without any changes. However, the implication $\textit{(2)}$ $\Rightarrow$ $\textit{(1)}$ is not true in general. A sufficient additional condition is the following one: for each $i$ the process
$(X_t^i)_{t \in [0,\infty)}$ is a generated $Q$-martingale generated by some $X_{\infty}^i$, i.e. 
$X_t^i=E_Q[X_{\infty}^i |\mathcal{F}_t ]$, $\ t \in [0,\infty)$. Then, the proof runs as before with $X_{\infty}^i$ taking the role of $X_T^i$.
\end{remark}

\begin{definition}\label{fe}
 In the context of Theorem \ref{darstellung} we call $$(\Omega, \mathcal{F}, \{Q_1,...,Q_d, Q, P\})$$ together with the filtration $(\mathcal{F}_t)_{t \in I}$ a financial experiment. 
The processes 
$\left( \dfrac{dQ_{i|\mathcal{F}_t}}{dQ_{|\mathcal{F}_t}}\right)_{t \in I} $ are called filtered likelihood processes (or density processes in finance).
\end{definition}

\begin{remark}\label{experiment}
  $\textit{(a)}$  We see that our price processes are completely determined by this experiment. If the real world measure is of no importance, for instance for various pricing formulas, we refer to $\{Q_1,...,Q_d, Q\})$  as the financial experiment.\\
  $\textit{(b)}$ Recall that, up to equivalence of statistical experiments (see Torgersen \cite{Torgersen:91} for the definition), 
our statistical experiment (given by $P$-dominated distributions) is uniquely determined by the likelihood ratio distribution
\begin{equation}\label{sm}
\mathcal{L}\left(\left(\frac{dQ_1}{dQ},...,\frac{dQ_d}{dQ},\frac{dP}{dQ}\right)\Big|P\right).
\end{equation}
Observe that it is easy to see that \eqref{sm} also uniquely determines the likelihood distribution of $\left(\frac{dQ_1}{dQ},...,\frac{dQ_d}{dQ},\frac{dP}{dQ}\right)$ under $Q_1,...,Q_d,Q$.
In consequence, our financial model is given by
$$\mathcal{L}\left(\left(\frac{dQ_1}{dQ},...,\frac{dQ_d}{dQ}\right)\Big|P\right).$$
In this latter setup there is room for one degree of freedom given by the martingale measure $Q$ which is not unique in general.\\
$\textit{(c)}$ If the martingale measure $Q$ is unique we have a one-to-one correspondence between financial models and financial experiments.
 
\end{remark}

\begin{example} [Cox-Ross-Rubinstein model] \label{crr}
 Consider the discrete time scale with $I=\{0,1,...,N\}$ and factors $u>d>0$ for upwards and downwards moving respectively. Without loss of generality we assume $\Omega=\{0,1\}^N$. For $\omega=(\omega_1,...,\omega_N)$ let $k_n=k_n(\omega)=\sum_{j=1}^n \omega_j$. The bond price process $S^0$ and price process of the risky asset $S^1$ can be written as
\begin{align*}
S_n^0=p_0 r^n, \quad n \in I, \quad
S_n^1=p_1 u^{k_n} d^{n-k_n}, \quad n \in I, \ \ 0 \le k_n \le n,
\end{align*}
where $p_0>0$, $r \ge 1$ and $p_1>0$ are constants. By $(\mathcal{F}_n)_{n \in I}$ we denote the canonical filtration.
The discounted price process is just
\begin{equation}
 X_n:=\frac{S_n^1}{S_n^0}=\frac{p_1}{p_0} \tilde{u}^{k_n} \tilde{d}^{n-k_n}, \quad n \in I, \ \ 0 \le k_n \le n
\end{equation}
with respect to $(\mathcal{F}_n)_{n \in I}$, where $\tilde{u}:=\frac u r$ and $\tilde{d}:=\frac d r$.\\
In this example we now get a short proof of the difficult half of the first fundamental asset pricing theorem (see for instance Shiryaev \cite{Shiryaev:99}, p. 417/418, 655), i.e. we show that the assumption of no arbitrage opportunities implies the existence of an equivalent martingale measure. Under the assumption that there are no arbitrage opportunities in the market we must have $u>r>d$, as is easy to check by using a straightforward contradiction argument, or equivalently $\tilde{u}>1>\tilde{d}$. Thus, by setting $\tau:=\dfrac{1-\tilde{d}}{\tilde{u}-\tilde{d}}$ and $\kappa:=\tau \tilde{u}$ we obtain $\tau, \kappa \in (0,1)$. Now $Q:=\left( (1-\tau)\varepsilon_0+\tau \varepsilon_1\right)^N$ and $Q_1:=\left( (1-\kappa)\varepsilon_0+\kappa \varepsilon_1\right)^N$ are product probability measures equivalent to $P$ satisfying
$$\frac{dQ_{1|\mathcal{F}_n}}{dQ_{|\mathcal{F}_n}} = \left(\frac{\kappa}{\tau}\right)^{k_n} \left(\frac{1-\kappa}{1-\tau}\right)^{n-k_n}=\tilde{u}^{k_n} \tilde{d}^{n-k_n}= \frac{X_n}{X_0}.$$
We conclude from Theorem \ref{darstellung} that $Q$ is a martingale measure.\\
Similarly, more general models - say trinomial models - can be discussed. Theorem \ref{satz1} below offers a method how to find all martingale measures if one of them is known.
\end{example}

Below we will use known results from the theory of statistical experiments in order to revisit financial models. In this paper we treat option prices, regression models, completeness and convergence of option prices.

\section{Option prices in terms of power of tests}\label{section3}

As first application we will derive option prices of certain options $H$ in terms of tests for statistical experiments. In addition to our price processes we consider a bond
\begin{equation}\label{bond}
 S_t^0=\exp\left(\int_0^t \rho(u) du\right)
\end{equation}
given by an integrable deterministic interest rate $\rho:[0,T] \rightarrow \mathbb{R}$. A lot of claims $H$ are paid out at the endpoint $T$. We consider only claims of this type throughout the paper. Typically, a price $p_Q(H)$ of an option with payoff $H$ like this is given by
\begin{equation}\label{ph}
p_Q(H)=E_Q((S_T^0)^{-1}H)
\end{equation}
where $Q$ is a martingale measure, see for instance Karatzas and Shreve \cite{Karatzas and Shreve:91} (p.378).
\bigskip\\
As motivation consider for $d=1$ and $S_0^1=s_0^1$ a European call option with strike price $K$ given by
 \begin{equation}\label{call}
H_C=(S_T^1-K)^+=(S_T^1-K)\ind_{\{S_T^1>K\}}.  
 \end{equation}
The indicator function 
$\ind_{\{S_T^1>K\}}=\ind_{\left\{\frac{X_T^1}{X_0^1}>K (s_0^1)^{-1} \exp\left(-\int_0^T\rho(u) du \right) \right\}}$
may be regarded as a test for our financial experiment which links option prices to power function of tests for statistical hypotheses.
\smallskip

$\textit{(A)}$ Assume that $H$, the payoff at time $T$, has the form
\begin{equation}\label{a}
 H=\left(\sum_{i=1}^d a_iS_T^i - K\right)\phi\left( \tilde{S} \right)
\end{equation}
with $\tilde{S}:=\left(\left( \frac{S_t^i}{s_0^i \exp\left( \int_0^t \rho(u)du\right) }\right)_{t \le T}\right)_{i=1,...,d}$
where $a_i$ and $K$ are real coefficients and $\phi:\left(\mathbb{R}^{[0,T]}\right)^d \rightarrow [0,1]$ can be viewed as a test.\\
In this formula $K$ is a strike price, $ \sum_{i=1}^d a_i S_T^i$ can be regarded as a new financial product and $\phi$ (mostly an indicator function) models the constraints which are responsible for the payoff $H \neq 0$. 
By linearity Theorem \ref{preis} below also holds for a finite sum
\begin{equation}\label{a2}
 \sum_{j=1}^m H_j
\end{equation}
where $H_j$ is as in \eqref{a}.
Examples for options satisfying condition $\textit{(A)}$ are the European call and the European put option as well as many other options like the Straddle option, the Strangle option, the Bull-Spread option 
(see Korn and Korn \cite{Korn and Korn:01}, p.148/149 for payoff profiles), which all depend only on the final value of the underlying asset or digital options and barrier options which may depend on the whole path.
Another example of a payoff function of type \eqref{a} is
\begin{equation}\label{a3}
 H=(a(S^1_T-S^2_T)-K) \ind_{\{S^1_T>c S^2_T\}}.
\end{equation}

\begin{theorem}\label{preis}
 Under assumption $\textit{(A)}$ and for a fixed martingale measure $Q$, the option price (\ref{ph}) of $H$ is given by
\begin{equation}\label{preisformel}
 p_Q(H)=\sum_{i=1}^d a_i s_0^i E_{Q_i}( \phi)-\exp\Big(-\int_0^T \rho(u)du\Big) K E_Q(\phi)
\end{equation}
where $\phi=\phi\left(\left( \dfrac{dQ_{1|\mathcal{F}_t}}{dQ_{|\mathcal{F}_t}} \right)_{t \le T},...,\left( \dfrac{dQ_{d|\mathcal{F}_t}}{dQ_{|\mathcal{F}_t}} \right)_{t \le T}\right)$. 
By linearity \eqref{preisformel} can be extended to finite sums \eqref{a2}.
\end{theorem}

\begin{remark}
$\textit{(a)}$ 
 The option price formula demonstrates the importance of financial experiments. 
The value $p_Q(H)$ is a linear function of power functions of tests $\phi$ given by the filtered experiments and their likelihood functions which yields another look on prices by purely statistical quantities.\\
$\textit{(b)}$ A similar price formula was obtained by Gerber and Shui \cite{Gerber and Shiu:94} via Esscher transforms for special stocks when $\log S(t)$ has stationary and independent increments.
\end{remark}

\begin{proof}
 Our proof starts with the observation that we can rewrite the price processes in a convenient form. Indeed, Theorem \ref{darstellung} yields 
$$S_t^i= S_t^0 X_t^i  = \exp\left( \int_0^t \rho(u)du \right) s_0^i \dfrac{dQ_{i|\mathcal{F}_t}}{dQ_{|\mathcal{F}_t}}.$$
Using this and the pricing formula \ref{ph}, it follows that
\begin{eqnarray*}
 p_Q(H)&=&E_Q\left( \exp\left( -\int_0^T \rho(u)du\right) H\right) \\
&=&\sum_{i=1}^d a_i E_Q\left(s_0^i \frac{dQ_{i|\mathcal{F}_T}}{dQ_{|\mathcal{F}_T}}\phi\right)  -\exp\Big( -\int_0^T \rho(u)du\Big) K E_Q( \phi)   \\
&=&\sum_{i=1}^d a_i s_0^i E_{Q_i}( \phi)-\exp\Big(-\int_0^T \rho(u)du\Big) K E_Q(\phi),
\end{eqnarray*}
which is the desired formula.
\end{proof}

\begin{example}[European call option]\label{calloption} 
 For a fixed martingale measure $Q$ the price of the European call (\ref{call}) is given by
\begin{equation}\label{pqhc}
 p_Q(H_C)=s_0^1 E_{Q_1}\left( \phi\left( \frac{dQ_1}{dQ}\right) \right)
-\exp\left( -\int_0^T \rho(u)du\right)K E_Q\left( \phi\left( \frac{dQ_1}{dQ}\right) \right),
\end{equation}
where $\phi\left( \dfrac{dQ_1}{dQ}\right)
=\ind_{\left\{\frac{dQ_1}{dQ}>c \right\}}$ can be identified as a Neyman Pearson test with critical value
$c:=\frac K{s_0^1}\exp\left(-\int_0^T\rho(u) du \right)$ for the null hypothesis $\{Q\}$ versus $\{Q_1\}$. 
We briefly summarize the following statistical results, see Lehmann and Romano \cite{Lehmann and Romano:05}, p.14 or Witting \cite{Witting:85}, p. 228.
The test $\phi$ is a Bayes test for the prior $\Lambda_0=\frac c{1+c}$ and $\Lambda_1=\frac 1{1+c}$ for testing $\{Q\}$ versus $\{Q_1\}$ 
with corresponding minimal Bayes risk of the test 
$$\frac{s_0^1-p_Q(H_C)}{s_0^1+K \exp\left(- \int_0^T \rho(u)du\right)}.$$
Similar results hold for testing $\{Q_1\}$ versus $\{Q\}$ and $\psi:=1-\phi$.\\
In case of the classical Black-Scholes model $\mathcal{L}\left( \log \frac{dQ_1}{dQ}|Q\right)=N(-\frac{\sigma^2}2T,\sigma^2 T) $ and 
$\mathcal{L}\left( \log \frac{dQ_1}{dQ}|Q_1\right)=N(\frac{\sigma^2}2T,\sigma^2 T) $ are normal distributions, see also Example \ref{ito}. Thus \eqref{pqhc} is just the classical Black-Scholes price of the European call which is treated again in Corollary \ref{bscall} below.
Similarly, the option price of \eqref{a3} is a linear combination of the power of suitable Neyman Pearson tests.
\end{example}

To explain the correspondence between finance and statistics consider the nonparametric regression model
\begin{equation}\label{regrmodel}
\xi(t):=W(t)+\int_0^t \gamma(u)du \quad,~ 0 \le t \le T, 
\end{equation}
with unobservable Brownian motion $W$ as noise and deterministic signal given by $\gamma \in \tilde{\Theta}:=L_2[0,T]$. Their distributions form a Gaussian shift which is just a typical nonparametric limit model, for instance in survival analysis. 
Here $\gamma$ has a proper interpretation as hazard rate derivative, see Janssen and Milbrodt \cite{Janssen and Milbrodt:93} (2.20), (3.20). 
A similar structure shows up for It\^{o} type models in finance where the parameter space 
$\tilde{\Theta}$ is extended to various adapted processes linked to the underlying volatility. 
Then, the price formula in Theorem \ref{preis} has an interpretation for tests and power functions for the new regression model.
The next example explains this relation between It\^{o} type price processes and the appertaining financial experiments.

\begin{example}[It\^o type price processes and regression models] \label{ito} 
For $I=[0,T]$ consider discounted price processes 
\begin{equation}
X_t^i=\frac{S_t^i}{S_t^0}=X_0^i \exp\left( \int_0^t \sigma_i'(s) dW(s) + \int_0^t \left( \mu_i(s)-\rho(s) - \frac{\|\sigma_i(s)\|^2} 2\right)ds \right)  
\end{equation}
where $W$ is a $d$-dimensional Brownian motion with respect to the real world measure $P$.\\
We specify a parameter space $\Theta$ by all volatility matrices $\sigma=(\sigma_{ij})_{i,j=1,...,d}$ 
under the usual conditions: progressive measurability, uniformly positive definiteness (see Korn and Korn \cite{Korn and Korn:01}, p. 57 for the definition) 
and $\sum_{i,j=1}^d \int_0^T \sigma_{ij}^2(u) du<\infty$ uniformly in $\omega$.
Introduce the $d$-dimensional column vector $\sigma_i := (\sigma_{i1},...,\sigma_{id})'.$\\
Let $\rho \ge 0$ and $\mu=(\mu_{1},...,\mu_{d})'$ denote progressively measurable, uniformly bounded processes which specify the interest rate $\rho$ and drift $\mu$, respectively, 
with the bond modeled by \eqref{bond}. At a first glance the reader should view $\sigma, \rho, \mu$ as deterministic functions. 
But there are good reasons to deal with process valued parameters. This special point is discussed at the end of this example.\\
Introduce $\ind=(1,...,1)' \in \mathbb{R}^d$ and set $\theta(s):=\sigma^{-1}(s)[\rho(s)\ind - \mu(s)]$.
By Girsanov's theorem (see Karatzas/Shreve \cite{Karatzas and Shreve:91}, p.191 Theorem 5.1) introduce $Q$: 
$$\frac{dQ}{dP}:=\exp\left(\int_0^T \theta'(s) dW(s) -\frac{1}{2} \int_0^T \|\theta(s)\|^2 ds\right), \quad
\bar{W}(t):=W(t)-\int_0^t \theta(s) ds.$$
$Q$ is a martingale measure and $\bar{W}$ is a $d$-dimensional Brownian motion with respect to $Q$.
Theorem \ref{darstellung} yields 
\begin{align}\label{ito2}
\frac{dQ_{i |\mathcal{F}_t}}{dQ_{|\mathcal{F}_t}}=\exp\left(\int_0^t \sigma_i'(s) d\bar{W}(s) -
\frac{1}{2} \int_0^t \|\sigma_i(u)\|^2 du\right)=\frac{X_t^i}{X_t^0}. 
\end{align}
We will now show how the financial experiment is linked to a $d$-dimensional regression model 
$\xi(t)=(\xi_1(t),...,\xi_d(t))', \quad 0 \le t \le T.$\\
Consider a progressively measurable vector $\tau=(\tau_1,...,\tau_d)'$. 
Based on the parameters $\sigma \in \Theta$, $\mu, \rho, \tau$ the regression model is defined by
\begin{equation}\label{regression}
  \xi(t)=W(t)+ \int_0^t \left(\tau(s) - \theta(s) \right)ds = \bar{W}(t)+ \int_0^t \tau(s) ds, \quad 0 \le t \le T.
 \end{equation}
For simplicity assume that $\Omega=C[0,T]^d$ with $d$-dimensional standard Wiener measure $P$ and identity $W$.
Then the process $(\xi(t))_{t\le T}$ has the distribution $Q$ for $\tau=0$ and $Q_i$ whenever $\tau=\sigma_i$, $i=1,...,d$.
Thus, the appertaining financial experiment is given by the distributions of a regression model when the parameters $\sigma, \mu, \rho$ are deterministic. 
When the parameters are processes we arrive at a generalized regression model. In this case the process
$$\xi(t)=\bar{W}_t+\int_0^t \tau(s) ds =:M_t+A_t$$ decomposes in the $Q$-martingale $\bar{W}_t=M_t$ and the compensator $A_t$ where now $(A_t)_{0 \le t \le T}$ serves as the parameter. 
In statistics that type of generalized models with random predictable parameters and random martingale parts show up for instance for counting processes in survival analysis, see  Andersen, Borgan, Gill and Keiding \cite{Andersen and Borgan and Gill and Keiding:97}.
\end{example}

\section{Changing martingale measures, completeness of markets and the dynamics of price processes}\label{section4}
In this section we will explain the change of martingale measures in terms of statistical experiments, see Remark \ref{experiment}. 
Also the dynamics of option prices in time is studied in terms of financial experiments.
As an application we can characterize the completeness of financial markets, which is connected in some way to the completeness of statistical experiments. We begin by presenting a concept for statistical experiments which factorizes the likelihood ratio via information which is available for a sub-$\sigma$-field $\mathcal{H} \subset \mathcal{F}$. This subject is of its own interest for statistical experiments.\\
In the sequel, consider a dominated experiment $E=(\Omega, \mathcal{F}, \{P_{\theta} : \theta \in \Theta \})$ and a $\sigma$-field $\mathcal{H} \subset \mathcal{F}$. Without loss of generality we can assume the existence of a $\theta_0$ for which $P_{\theta} \ll P_{\theta_0}$ for all $\theta \in \Theta$ since otherwise we could just add a certain $P_{\theta_0}$, which satisfies $P_{\theta} \ll P_{\theta_0}$, to the experiment (see Lehmann and Romano \cite{Lehmann and Romano:05}, p.699 Theorem A.4.2 or Torgersen \cite{Torgersen:91}, p.6). We introduce the restricted experiment $E_{|\mathcal{H}}$ by setting
$$E_{|\mathcal{H}}:=(\Omega, \mathcal{H}, \{P_{\theta|\mathcal{H}} : \theta \in \Theta \}).$$

Some calculations now give us the following elementary lemma.

\begin{lemma}\label{faktorisierung}
 We can factorize the likelihood of the experiment $E$ into the likelihoods of $E_{|\mathcal{H}}$ and an experiment $E'$, called complementary to $E_{|\mathcal{H}}$ with respect to $E$, 
i.e. we have an experiment $E':=(\Omega, \mathcal{F}, \{P'_{\theta} : \theta \in \Theta \})$ with $P'_{\theta_0}=P_{\theta_0}$, such that
\begin{itemize}
 \item[(a)] $\dfrac{dP_{\theta|\mathcal{H}}}{dP_{\theta_0|\mathcal{H}}}=E_{\theta_0}\left[\dfrac{dP_\theta}{dP_{\theta_0}} \Big|\mathcal{H} \right]$ and
 \item[(b)] the densities of $E$ factorize into densities of $E'$ and $E_{|\mathcal{H}}$, i.e.
\begin{equation}\label{dichte} 
\frac{dP_\theta}{dP_{\theta_0}}=\frac{dP'_\theta}{dP_{\theta_0}} \cdot \dfrac{dP_{\theta|\mathcal{H}}}{dP_{\theta_0|\mathcal{H}}}.
\end{equation}
\end{itemize}
\end{lemma}

\begin{remark}\label{aequi}
$\textit{(a)}$ Statement \eqref{dichte} implies $Cov_{P_{\theta_0}}\left(\dfrac{dP_{\theta|\mathcal{H}}}{dP_{\theta_0|\mathcal{H}}},\dfrac{dP'_\theta}{dP_{\theta_0}}\right)=0$ but nothing about independence in general. 
Note that the covariance can be written as formal expression by \eqref{dichte} without further assumptions about square integrability.
The product structure of the densities does not imply any independence since the dominating measure is no product measure in general.\\
Formula \eqref{dichte} has an interpretation as a factorization of a joint density $f(X,Y)$ into a conditional density $f(X|Y)$ of $X$ given $Y$ and a marginal density $g(Y)$ of $Y$, i.e. $f(X,Y)=f(X|Y)g(Y)$. 
To see this consider the identities $X=id: (\Omega, \mathcal{F}) \rightarrow (\Omega, \mathcal{F}), Y=id: (\Omega, \mathcal{F}) \rightarrow (\Omega, \mathcal{H})$ where $Y$ only collects the information given by $\mathcal{H}$. 
By the sufficiency of $X$ we have 
$$\dfrac{dP_{\theta}}{dP_{\theta_0}}=\dfrac{d\mathcal{L}(X|P_{\theta})}{d\mathcal{L}(X|P_{\theta_0})}
=\dfrac{d\mathcal{L}((X,Y)|P_{\theta})}{d\mathcal{L}((X,Y)|P_{\theta_0})}=:f(X,Y).$$
On the other hand
$$\dfrac{dP_{\theta|\mathcal{H}}}{dP_{\theta_0|\mathcal{H}}}=E_{\theta_0}\left[\dfrac{dP_\theta}{dP_{\theta_0}} \Big|\mathcal{H} \right]
=\dfrac{d\mathcal{L}(Y|P_{\theta})}{d\mathcal{L}(Y|P_{\theta_0})}=:g(Y).$$
Thus $\dfrac{dP'_\theta}{dP_{\theta_0}}=:f(X|Y)$ corresponds to the conditional density of $X$ given $Y$ at least when the conditional distribution of $X$ given $Y$ allows a desintegration w.r.t. a regular conditional distribution 
$P_{\theta_0}(X \in \cdot ~|Y)$. Note that this assumption is only needed for the interpretation; Lemma \ref{faktorisierung} works without this additional assumption.\\
$\textit{(b)}$ If $\dfrac{dP_{\theta|\mathcal{H}}}{dP_{\theta_0|\mathcal{H}}}$ and $\dfrac{dP'_\theta}{dP_{\theta_0}}$ are $P_{\theta_0}$ independent, then it is easy to see that the binary experiments
$\{P_{\theta},P_{\theta_0}\}$ and $\{P_{\theta|\mathcal{H}}\otimes P'_{\theta},P_{\theta_0|\mathcal{H}}\otimes P_{\theta_0}\}$ are equivalent in LeCam's sense, i.e. the distributions of the likelihood ratios are the same.
\end{remark}

\hspace{-.63cm}
\textbf{Proof of Lemma \ref{faktorisierung}:}
 A trivial verification shows $\textit{(a)}$. For the construction of the experiment $E'$ we prove that the left hand side of \eqref{dichte} can be divided by $\dfrac{dP_{\theta|\mathcal{H}}}{dP_{\theta_0|\mathcal{H}}}$. 
To this end, for a fixed $\theta \in \Theta$ set 
$A_\theta := \left\{E_{\theta_0}\left[\frac{dP_\theta}{dP_{\theta_0}} \Big|\mathcal{H} \right]=0\right\}.$
With this notation we get $A_\theta \in \mathcal{H}$ and 
\begin{eqnarray*}
	P_\theta(A_\theta)
	&=& \int_\Omega \frac{dP_\theta}{dP_{\theta_0}} 1_{A_\theta} dP_{\theta_0}
 = \int_\Omega E_{\theta_0}\left[\frac{dP_\theta}{dP_{\theta_0}} 1_{A_\theta}\Bigg|\mathcal{H} \right]dP_{\theta_0}\\
	&=& \int_\Omega 1_{A_\theta} E_{\theta_0}\left[\frac{dP_\theta}{dP_{\theta_0}} \Bigg|\mathcal{H} \right]dP_{\theta_0}=0.
\end{eqnarray*}
Thus, $\frac{dP_\theta}{dP_{\theta_0}} 1_{A_\theta} =0$ $P_{\theta_0}$-almost surely.
Using the usual convention $0 \cdot \infty =0$, this yields
\begin{eqnarray*}
\frac{dP_\theta}{dP_{\theta_0}}=E_{\theta_0}\left[\frac{dP_\theta}{dP_{\theta_0}} \Bigg|\mathcal{H} \right] \cdot f_\theta
\ \textnormal{, where} \ \
f_\theta:=\frac{dP_\theta}{dP_{\theta_0}} \Big/ E_{\theta_0}\left[\frac{dP_\theta}{dP_{\theta_0}} \Bigg|\mathcal{H} \right].
\end{eqnarray*}
It follows $E_{\theta_0}[f_\theta|\mathcal{H} ]=1$ and consequently $E_{\theta_0}(f_\theta)=1$.
Therefore, by setting
$\frac{dP'_\theta}{dP_{\theta_0}}:= f_\theta$, we can define a probability measure $\ P'_\theta$ on $(\Omega, \mathcal{F})$ with $P'_{\theta_0}=P_{\theta_0}$.
Using $\textit{(a)}$, this gives us the desired equation \eqref{dichte}.\\
\hspace*{\fill}\begin{small}$\square$\end{small}

\begin{example}
Let $\mathcal{P}$ be the set of all martingale measures equivalent to $P$.
For a fixed time $t \in I \subset [0,T]$ and a fixed $Q \in \mathcal{P}$ consider the financial experiment $\{Q_1,...,Q_d,Q\}$ together with the filtration $(\mathcal{F}_t)_{t \in I}$ as above.\\
$\textit{(a)}$ The restricted experiment 
$E_{|\mathcal{F}_t}=(\Omega, \mathcal{F}_t, \{Q_{1|\mathcal{F}_t},...,Q_{d|\mathcal{F}_t},Q_{|\mathcal{F}_t}\})$
admits a complementary experiment $E'_t=\{Q'_1(t),...,Q'_d(t),Q(t)\}$ with $Q(t)=Q$.\\
   $\textit{(b)}$ If we turn our attention to the densities of the new experiment $E'_t$, we conclude from \eqref{dichte} and \eqref{likelihood} that 
\begin{equation}\label{komp}
\frac{dQ'_i(t)}{dQ}=\frac{X^i_T}{X^i_t} \quad \textnormal{ and} \quad E_{Q(t)}\left[\frac{dQ'_i(t)}{dQ} \Bigg|\mathcal{F}_{t+s} \right]=\frac{X^i_{t+s}}{X^i_t}, \quad 0 \le s \le T-t
\end{equation}
since $Q \in \mathcal{P}$.
Thus, the new experiment $E'_t$ has a concrete meaning. Whenever the time $t$ is over and the price processes are observed until time $t$, 
the experiment $E'_t$ describes the updated and normalized price process 
$$I \cap [0,T-t]\ni s \mapsto \frac{X^i_{t+s}}{X^i_t} \quad \textnormal{and }~ S^i_{t+s}=S^i_t \exp\left(\int_t^{t+s}\rho(u)du\right)\frac{dQ'_i(t)_{|\mathcal{F}_{t+s}}}{dQ_{|\mathcal{F}_{t+s}}}.$$
\end{example}

\begin{example}[Dynamics of option prices]
The complementary experiment can be used to describe the dynamics of option prices. 
We will illustrate this for the European call option $H_C$ from Example \ref{calloption}. 
For this purpose we may write
$$\phi\left(\frac{dQ_1}{dQ}\right)=\phi\left(\frac{X_t^1}{X_0^1}\cdot\frac{X_T^1}{X_t^1}\right)
=\phi\left(\frac{X_t^1}{X_0^1}\cdot\frac{dQ'_1(t)}{dQ}\right).$$
At time $t$ consider the value $S_t^1=s_t^1$ and introduce a new updated test by
$$\phi_t(x,s):=\phi\left(  \frac{s x}{s_0^1 \exp\left( \int_0^t \rho(u)du\right) } \right)$$
which implies
$$\phi_t\left( \frac{dQ'_1(t)}{dQ},s_t^1\right)
=\phi\left(  \frac{s_t^1}{s_0^1 \exp\left( \int_0^t \rho(u)du\right) } \cdot\frac{dQ'_1(t)}{dQ} \right).$$
If the prices are calculated via the martingale measure $Q$ the price $p_Q(H_C,t)$ of $H_C$ at time $t$ in the classical 
Black-Scholes model with constant parameters $\sigma>0, \mu \in \mathbb{R}$ and $\rho>0$ (see Example \ref{ito}) can be described in terms of the 
complementary experiment similarly to \eqref{pqhc} by
\begin{align*}\label{pqhct}
 p_Q(H_C,t)&=s_t^1 E_{Q'_1(t)}\left( \phi_t\left( \frac{dQ'_1(t)}{dQ},s_t^1\right) \right)\\
&~~~-\exp\left( -\rho(T-t)\right)K E_Q\left( \phi_t\left( \frac{dQ'_1(t)}{dQ},s_t^1\right) \right).
\end{align*}
For complete markets this price is the unique price at time $t$.
This method can also be used in a more general setup where the payoff $H$ of the form \eqref{a} only depends on the last price at time $T$ via $\phi$.
Then, the option price of $H$ at time $t$ looks similar to \eqref{preisformel} and can be expressed with the help of the complementary experiment.
The measures $Q'_i(t)$ can be interpreted as updates of the measures $Q_i$ when time $t$ has passed.
\end{example}

The next lemma will be needed in the remainder of this section.
\begin{lemma}\label{hl1}
Fix two probability measures $P_0,P_1$ on $(\Omega, \mathcal{F})$ with $P_0 \ll P_1$. \\
Furthermore, fix $f \in L_1(P_0)$  and a $\sigma$-field $\mathcal{H} \subset \mathcal{F}$. Then
$$E_{P_1}\left[f\frac{dP_0}{dP_1} \Bigg|\mathcal{H} \right]=E_{P_0}[f | \mathcal{H}] E_{P_1}\left[\frac{dP_0}{dP_1} \Bigg|\mathcal{H} \right].$$
\end{lemma}
The proof is well known and left to the reader.
The next theorem establishes a nice ``martingale measure criterion'' in terms of the complementary experiment.

\begin{theorem}\label{satz1}
 For $Q \in \mathcal{P}$, assume that $Q^{\ast}$ is another probability measure equivalent to $Q$ with $g :=\frac{dQ^{\ast}}{dQ}$. Then the following conditions are equivalent:
\begin{enumerate}
   \item[(1)]  For all $1 \le i \le d$ and $t \in I$:
$E_{Q'_i(t)}[g|\mathcal{F}_t]=E_Q[g|\mathcal{F}_t]$.
   \item[(2)] The probability measure $Q^{\ast}$ is a martingale measure, i.e. $Q^{\ast} \in \mathcal{P}$.
\end{enumerate}
\end{theorem}

\begin{proof}
 Fix $1 \le i \le d$ and $t \in I$. From \eqref{dichte} and $Q \in \mathcal{P}$ we obtain
$$E_Q\left[\frac{dQ'_i(t)}{dQ} \Bigg|\mathcal{F}_t \right]=E_Q\left[\frac{X^i_T}{X^i_t} \Bigg|\mathcal{F}_t \right]=1.$$
Combining this fact with Lemma \ref{hl1} (using the choices $f:=g$, $P_0:=Q'_i(t)$ and $P_1:=Q$), the statement \eqref{dichte} yields
$$E_{Q'_i(t)}[g|\mathcal{F}_t]= E_Q\left[g \frac{dQ'_i(t)}{dQ} \Bigg|\mathcal{F}_t \right]
	= E_Q\left[g \frac{X^i_T}{X^i_t} \Bigg|\mathcal{F}_t \right] = \frac{1}{X^i_t}E_Q[g X^i_T|\mathcal{F}_t].$$
Hence, $\textit{(1)}$ holds if and only if
\begin{eqnarray*}
 E_Q[g X^i_T|\mathcal{F}_t]=X^i_t E_Q[g|\mathcal{F}_t].
\end{eqnarray*}
Applying Lemma \ref{hl1} again (this time with $f:=X^i_T$, $P_0:=Q^{\ast}$ and $P_1:=Q$) leads to
$$E_Q[g X^i_T|\mathcal{F}_t]=E_{Q^{\ast}}[X^i_T|\mathcal{F}_t] E_Q[g|\mathcal{F}_t].$$
Thus, $\textit{(1)}$ is equivalent to $E_{Q^{\ast}}[X^i_T|\mathcal{F}_t] E_Q[g|\mathcal{F}_t]=X^i_t E_Q[g|\mathcal{F}_t]$ or 
$E_{Q^{\ast}}[X^i_T|\mathcal{F}_t]=X^i_t$ for all $1 \le i \le d$ and $t \in I$, respectively. This is the desired conclusion.
\end{proof}

We see that, whenever the structure of the financial experiments $E'_t$ is not rich enough, there may exist more than one martingale measure (i.e. there is more than one function $g$ satisfying $\textit{(1)}$ above). Now we show how this is linked to the notion of completeness of classes of statistical experiments.

\begin{definition}[Completeness of classes of experiments]\label{vollstaendig}
Assume that $\{P_{\theta} : \theta \in \Theta \}$ is a dominated experiment and that there is a
$\theta_0 \in \Theta$ satisfying $P_{\theta} \ll P_{\theta_0}$ for all $\theta \in \Theta$. \\
Consider a class of measurable functions $\mathcal{G} \subset \bigcap_{\theta \in \Theta} L_1(P_{\theta})$ . \\
We call $\mathcal{G}$ complete with respect to $\{P_{\theta} : \theta \in \Theta \}$ if, for every $g \in \mathcal{G}$ satisfying $E_{\theta}(g) = E_{\theta_0}(g)$ for all $\theta \in \Theta$ we have, that $g$ is constant $P_{\theta_0}$-almost surely.
\end{definition}

Like mentioned earlier, the assumption $P_{\theta} \ll P_{\theta_0}$ is not particularly restrictive. \\
Consider a financial experiment $E=\{Q_1,...,Q_d,Q\}$.
Application of Theorem \ref{satz1} now yields the following result.

\begin{theorem}\label{satz2}
 Let $\mathcal{G}$ be the subset of strictly positive functions of $~\bigcap_{i=1}^d L_1(Q_i) \cap L_1(Q)$ where $Q \in \mathcal{P}$ is as before.
 The following conditions are equivalent:
\begin{enumerate}
   \item[(1)] The set of equivalent martingale measures is a singleton, i.e. $\mathcal{P}=\{Q\}$.
   \item[(2)] If $g \in \mathcal{G}$ and if
\begin{equation}\label{dichtebedingung}
 E_{Q'_i(t)}[g|\mathcal{F}_t]=E_Q[g|\mathcal{F}_t]
\end{equation}
holds for all $1 \le i \le d$ and $t \in I$,
then $g$ is constant $Q$-almost surely.
\end{enumerate}
\end{theorem}

\begin{proof}
 Let us first proof $\textit{(1)}$ $\Rightarrow$ $\textit{(2)}$.
Consider $g \in \mathcal{G}$ as in $\textit{(2)}$ above. We set $g^{\ast}:=\frac{g}{E_Q(g)}$ and then 
$Q^{\ast}(A):=\int_A g^{\ast} \; dQ $ for all $A \in \mathcal{F}$ defines a probability measure $Q^{\ast}$ equivalent to $Q$. From \eqref{dichtebedingung} we get
\begin{equation*}
 E_{Q'_i(t)}[g^{\ast}|\mathcal{F}_t]=E_Q[g^{\ast}|\mathcal{F}_t]
\end{equation*}
for all $1 \le i \le d$ and $t \in I$.
By applying \ref{satz1} we conclude that $Q^{\ast}$ is a martingale measure, i.e. $Q^{\ast} \in \mathcal{P}$. From $\textit{(1)}$ we obtain $Q^{\ast}=Q$. 
This gives $g^{\ast}=1$ $Q$-almost surely and therefore $g$ is constant $Q$-almost surely. \\
To prove $\textit{(2)}$ $\Rightarrow$ $\textit{(1)}$ fix $Q^{\ast} \in \mathcal{P}$. 
By setting $g:=\frac{dQ^{\ast}}{dQ}$ Theorem \ref{satz1} shows that \eqref{dichtebedingung} is satisfied for all $1 \le i \le d$ and $t \in I$. 
From $\textit{(2)}$ we deduce that $g$ is constant $Q$-almost surely, but since $E_Q(g)=1$ this forces $g=1$ $Q$-almost surely and hence $Q^{\ast}=Q$.
\end{proof}

\begin{remark}
$\textit{(a)}$ We consider the case where the set $I \subset [0,T]$ of possible trading times is a finite set. Recall from Shiryaev \cite{Shiryaev:99}, p.481 that in this case the condition $| \mathcal{P} |=1$ given in $\textit{(1)}$ in Theorem \ref{satz2} is necessary and sufficient for the completeness of the market, i.e. every claim can be replicated. This result is also known as the second fundamental asset pricing theorem.\\
   $\textit{(b)}$ We get a simple expression for condition $\textit{(2)}$ in Theorem \ref{satz2} in the case when only two times for trading are allowed, i.e. only $t=0$ and $t=T$ are possible. Then \eqref{dichte} yields $Q'_i(T)=Q$ and, in consequence, \eqref{dichtebedingung} always holds for $t=T$. By assumption $\mathcal{F}_0=\sigma(\mathcal{N})$ we get $Q'_i(0)=Q_i$ at the other time for trading $t=0$. Hence, we can rewrite \eqref{dichtebedingung} as $E_{Q_i}(g)=E_Q(g)$ for all $g \in \mathcal{G}$. 
Now $\textit{(2)}$ just restates the completeness of $\mathcal{G}$ with respect to $\{Q_1,...,Q_d,Q\}$.\\
In general, statement $\textit{(2)}$ in Theorem \ref{satz2} can only be interpreted as some sort of conditional completeness (in the sense of Definition \ref{vollstaendig}) for all $t$.
\end{remark}
By combining $\textit{(a)}$ and $\textit{(b)}$ of the previous remark we are now in the position to see that Theorem \ref{satz2} links both concepts of completeness.

\section{Convergence of option prices}\label{section5}
Our goal in this section is to establish convergence results for option prices using the price formulas of section \ref{section3} when the financial experiments are convergent. 
The application of Theorem \ref{limitp} and section \ref{section6} below require further structure results about experiments which are developed next.\\
Throughout, we will consider a financial experiment as in Definition \ref{fe} given by a fixed martingale measure $Q$ and a filtration $(\mathcal{F}_t)_{t \in I}$ 
where $I \subset [0,T]$ satisfying $\{0,T\} \subset I$ denotes a bounded time domain which may be discrete.\\
As a main tool for further results we will derive a standard representation of the financial experiment on the space $(([0,\infty]^d)^I,(\mathcal{B}([0,\infty]^d))^I)$ equipped with the product $\sigma$-field. 
Recall that quite different representations of statistical experiments can be equivalent in LeCam's sense. 
The notion of standard experiments (given by the law of the likelihood) is very useful and leads to better comparable forms, see Strasser \cite{Strasser:85a} and Torgersen \cite{Torgersen:91}.\\
For $J \subset I$ introduce canonical projections
\begin{equation}\label{projektion}
 \pi_J:([0,\infty]^d)^I \rightarrow ([0,\infty]^d)^J
\end{equation}
and as abbreviation write $\pi_t:=(\pi^i_t)_{i \le d}$ for a singleton $J=\{t\}$. The projections induce another filtration
\begin{equation}\label{filtration}
 \mathcal{G}_t:=\sigma(\pi_s : s \in I \cap [0,t]), \quad t \in I
\end{equation}
on $([0,\infty]^d)^I$. In addition for $t \in I$ we define $Y_t: \Omega \rightarrow ([0,\infty]^d)^{I \cap [0,t]}$ by
\begin{equation}\label{yt}
 Y_t:=\left( \left( \frac{X_s^i}{X_0^i}\right)_{i=1,...,d} \right)_{s \in I \cap [0,t]}.
\end{equation}
From Theorem \ref{darstellung} we know that $\dfrac{X_T^i}{X_0^i}=\dfrac{dQ_i}{dQ}$ and we conclude that $Y_T$ is a sufficient statistic for the experiment $\{Q_1,...,Q_d,Q\}$. Thus, we do not lose information if we turn to the experiment given by the image laws of $Y_T$, compare with Lehmann and Romano \cite{Lehmann and Romano:05} for the consequences of sufficiency. Define on $(([0,\infty]^d)^I,(\mathcal{B}([0,\infty]^d))^I)$
\begin{equation}
 \nu:=\mathcal{L}(Y_T | Q), \quad \nu_i:=\mathcal{L}(Y_T | Q_i).
\end{equation}
Within this setup a useful standard form of financial experiments can be established.

\begin{theorem}\label{bild}
 The image experiment $\{\nu_1,...,\nu_d,\nu\}$ together with the filtration $(\mathcal{G}_t)_{t \in I}$ is a filtered financial experiment with price process given by the projections $(\pi_t)_{t \in I}$, i.e.
\begin{equation}\label{darstellung2}
 \dfrac{d\nu_{i|\mathcal{G}_t}}{d\nu_{|\mathcal{G}_t}}=\pi_t^i, \quad i=1,...,d, \ \ t \in I.
\end{equation}
\end{theorem}

\begin{proof}
 In a first step, formula \eqref{darstellung2} will be proved for $t=T$. Combining \eqref{darstellung} and \eqref{yt} yields
\begin{equation}
 \dfrac{X_T^i}{X_0^i}=\dfrac{dQ_i}{dQ}=\pi_T^i \circ Y_T.
\end{equation}
Applying the transformation formula for image densities now gives us \eqref{darstellung2}.\\
To prove the general case let
\begin{equation}
 \nu_t:=\mathcal{L}(Y_t|Q)=\mathcal{L}(Y_t|Q_{|\mathcal{F}_t}), \quad \nu_t^i:=\mathcal{L}(Y_t|Q_i)=\mathcal{L}(Y_t|Q_{i|\mathcal{F}_t})
\end{equation}
denote the image laws on $([0,\infty]^d)^{I \cap [0,t]}$. The first step above implies
\begin{equation}\label{darstellung3}
 \dfrac{d\nu_t^i}{d\nu_t}=\pi_{J,t}^i
\end{equation}
where $\pi_{J,t}^i$ is the canonical projection given by \eqref{projektion} but with $J:=I \cap [0,t]$ instead of $I$. \\
Notice that $\mathcal{G}_t$ consists of all sets 
$$B=A \times ([0,\infty]^d)^{I \setminus J}, \quad A \in (\mathcal{B}([0,\infty]^d))^J$$
and that
\begin{equation}
 \nu(B)=\nu_t(A), \quad \nu_i(B)=\nu^i_t(A)
\end{equation}
holds for all $i=1,...,d$. In consequence, \eqref{darstellung3} leads to \eqref{darstellung2}.
\end{proof}

Since our goal is to get results for convergence, we require the concept of contiguity. Contiguity can be interpreted as the asymptotic form of absolute continuity (see for example Lehmann and Romano \cite{Lehmann and Romano:05}, p. 492), or, to be more precise: we call a sequence of probability measures $(P_n)$ contiguous with respect to $(Q_n)$ (where $P_n$ and $Q_n$ are defined on measurable spaces $(\Omega_n, \mathcal{F}_n)$) and write $P_n \triangleleft Q_n$ if for any sequence of sets $A_n \in \mathcal{F}_n$ with $Q_n(A_n) \rightarrow 0$ we have $P_n(A_n) \rightarrow 0$. 
Recall that by LeCam's first lemma it is well known that the contiguity $Q_n \triangleleft P_n$ is equivalent to the tightness of
$\mathcal{L}\left(\log \dfrac{dQ_n}{dP_n}\Big|Q_n\right)$ on $\mathbb{R}$.
Contiguity plays an important role in asymptotic statistics, where one famous result is the third lemma of LeCam (see Jacod and Shiryaev \cite{Jacod and Shiryaev:03}, p.621 Theorem 3.3 for a suitable reference), which we want to apply in the context of statistical experiments.
Contiguity has been put to use in finance, too. For instance, the application of contiguity (see Kabanov and Kramkov \cite{Kabanov and Kramkov:94}, p.184) leads to a criterion for asymptotic arbitrage in large financial markets 
and Hubalek and Schachermayer \cite{Hubalek and Schachermayer:98} discussed the convergence of option prices under contiguity. 
These authors used contiguity of a sequence of martingale measures w.r.t. to the physical measure. 
In contrast to this approach we use contiguity of the measures $Q_i$ and $Q$ of \eqref{darstellung} which contribute to the asset process. 
Recall that for pricing the physical measure is often of minor importance.\\
Consider now a sequence of financial experiments where the time domain $I$ is fixed but everything else in Theorem \ref{darstellung} and Definition \ref{fe} depends on the sequence index $n \in \mathbb{N}$. We write $Q_{1,n},...,Q_{d,n},Q_n$ on measurable spaces $(\Omega_n, \mathcal{F}_n)$ equipped with arbitrary filtrations $(\mathcal{F}_{t,n})_{t \in I}$ and price processes
$$\dfrac{dQ_{i,n|\mathcal{F}_{t,n}}}{dQ_{n|\mathcal{F}_{t,n}}}=\dfrac{X^i_{t,n}}{X^i_{0,n}}.$$\smallskip
Also attach the index $n$ at the vector (\ref{yt}) which we then denote by $Y_{t,n}$.\\
The next Theorem establishes a general convergence result for financial experiments. Of course, the weak compactness of classes of statistical experiments is well known. 
This result is now restated in terms of filtered financial experiments where the accumulation point has a standard form, see Theorem \ref{bild}.
Concrete examples and applications are studied in section \ref{section6}.

\begin{theorem}\label{konvergenz1}
 Consider a sequence $E_n=(\Omega_n, \mathcal{F}_n,\{Q_{1,n},...,Q_{d,n},Q_n\}), (\mathcal{F}_{t,n})_{t \in I}$ of financial experiments such that
$$Q_{i,n} \triangleleft \hspace{-.1cm} \triangleright \hspace{.1cm} Q_n, \quad (\text{ i.e. } Q_{i,n} \triangleleft Q_n \text{ and } Q_n \triangleleft Q_{i,n})$$
holds for all $i=1,..,d$.
Then there exists a standard financial experiment \\
$E=(([0,\infty)^d)^I,(\mathcal{B}([0,\infty)^d))^I,\{\nu_1,...,\nu_d,\nu\})$ given by the price process (\ref{darstellung2}) which is an accumulation point of the underlying sequence in the following sense: there exists a subnet $(\mathcal{I},\le)$ of $(\mathbb{N},\le)$ (see Kelley \cite{Kelley:75}, chapter 2 for the definition) such that all finite dimensional marginal distributions of $(\mathcal{L}(Y_{T,\tau}|Q_{i,\tau}))_{\tau \in \mathcal{I}}$ are weakly convergent to the marginal of $\nu_i$ for all $i=1,..,d+1$, where by convention $\nu_{d+1}=\nu$ and $Q_{d+1,n}=Q_n$.
\end{theorem}

\begin{remark}
The convergence in Theorem \ref{konvergenz1} is equivalent to the weak convergence in the sense of LeCam of the experiments
$\{Q_{i,\tau|\mathcal{F}_{t_j,\tau}}:1 \le i \le d+1, 1 \le j \le k\}$ to $\{\nu_{i|\mathcal{G}_{t_j}}:1 \le i \le d+1, 1 \le j \le k\}$ for all finite subsets $\{t_1,...,t_k\} \in I$.
This assertion is an immediate consequence of Strasser \cite{Strasser:85a}, p.302, Theorem 60.3.
\end{remark}

\hspace{-.63cm}
\textbf{Proof of Theorem \ref{konvergenz1}:}
 Let $\{\nu_{1,n},...,\nu_{d,n},\nu_{d+1,n}\}$ denote the financial standard experiment obtained from $\{Q_{1,n},...,Q_{d,n},Q_n\}$ by using Theorem \ref{bild} where $\nu_{d+1,n}$ belongs to the martingale measure $Q_n$. By an embedding argument we may assume that $\nu_{d+1,n}$ is defined on $(([0,\infty]^d)^I,(\mathcal{B}([0,\infty]^d))^I)$.\\
According to Lemma \ref{radon} of the appendix, $\nu_{d+1,n}$ extends to a uniquely determined Radon probability measure on the compact space $([0,\infty]^d)^I$.\\
Since the set of Radon probability measures on compact spaces is weakly compact, the sequence $(\nu_{d+1,n})_{n \in \mathbb{N}}$ has a weak cluster point $\nu$. Hence, there exists a subnet $(\mathcal{I},\le)$ such that $\nu_{d+1,\tau} \rightarrow \nu$ weakly along the subnet $\tau \in \mathcal{I}$.\\
We now turn to the finite dimensional marginals of the likelihood processes. Consider $t_j \in I$, $t_1<t_2<...<t_k=T$. For convenience, assume without restrictions that $X_{0,n}^i=1$ holds for all $i=1,...,d$. Then
\begin{equation}\label{p1}
 \mathcal{L}\left( (X^i_{t_j,\tau})_{(i,j) \in \{1,...,d\}\times\{1,...,k\}}|Q_{\tau}\right)  \rightarrow \mathcal{L}(\pi_J | \nu)
\end{equation}
weakly on $([0,\infty]^d)^J$ where $J=\{t_1,...,t_k\}$ and the projection $\pi_J$ in (\ref{projektion}) is extended to $([0,\infty]^d)^I$. Note that the distributions lie on $([0,\infty)^d)^J$ and that the sequence is tight by Markoff's inequality since
$$\int X^i_{t_{j,n}}dQ_n=1$$
by our assumptions for all $i,n$ and $t_j$. We also remark that a subsequence instead of a net can be used in this situation. \\
Thus, we have weak convergence in (\ref{p1}) on $([0,\infty)^d)^J$ to some distribution $\nu_J$ where $\nu_J$ is the restriction of $\mathcal{L}(\pi_J|\nu)$ to $([0,\infty)^d)^J$. Now $(\nu_J)_{J \subset I, |J|<\infty}$ forms a projective system and, consequently, uniquely determines the distribution $\nu=\nu_{d+1}$ on $([0,\infty)^d)^I$ required in Theorem \ref{konvergenz1}.\\
We proceed to establish convergence under $Q_{r,\tau}$ for each $r \le d$. At this point we apply the third lemma of LeCam (see Jacod and Shiryaev \cite{Jacod and Shiryaev:03}, p.621, Theorem 3.3). We claim that under the distributions $Q_{r,\tau}$
\begin{equation}\label{p2}
 \mathcal{L}\left( (X^i_{t_j,\tau})_{(i,j) \in \{1,...,d\}\times\{1,...,k\}}|Q_{r,\tau}\right)  \rightarrow \nu_{r,J}
\end{equation}
weakly on $([0,\infty]^d)^J$ where $$\dfrac{d\nu_{r,J}}{d\nu_J}\left( x^i_{t_j}\right)_{(i,j) \in \{1,...,d\}\times\{1,...,k\}} =x^r_T$$
and $J \subset I, |J|<\infty$.
Since the likelihood ratio
$$\dfrac{dQ_{r,\tau}}{dQ_{\tau}}=X^r_{T,\tau}$$
is part of (\ref{p2}), LeCam's third lemma establishes the desired result. Again, the system $(\nu_{r,J})_{J \subset I, |J|<\infty}$ forms a projective system and we obtain $\nu_r$ on $([0,\infty)^d)^I$ with $\dfrac{d\nu_r}{d\nu}=\pi^r_T$ where this density formula carries over from the marginals.\\
Now fix $t \in I$ and consider $I'=I \cap [0,t]$. By using the projection as in the proof of Theorem \ref{bild} we obtain
$$\dfrac{d\nu_{r|\mathcal{G}_t}}{d\nu_{|\mathcal{G}_t}}=\pi^r_t$$
and, in consequence, $(\nu_1,...,\nu_d,\nu)$ defines a standard financial experiment.
\hspace*{\fill}\begin{small}$\square$\end{small}\\
${}$\\
Observe that the standard form of financial experiments on $([0,\infty)^d)^I$ allows to fix the filtration \eqref{filtration}.

\begin{lemma}\label{teilfolge}
 The subnet above can be chosen as a subsequence in the following case:\\
There exists a countable subset $J \subset I$ such that each $\nu$ given as an arbitrary accumulation point of the limit experiment above is uniquely determined by its $J$-dimensional marginal distribution $\nu^{\pi_J}$. 
Thus, convergent subsequences exist when for instance $I=[0,T]$ and all accumulation points $\nu$ can be supported by the continuous functions $C[0,T]$.
\end{lemma}

\begin{proof}
 From the proof of Theorem \ref{konvergenz1} we know that the sequence $(\nu_{d+1,n})_{n \in \mathbb{N}}$ has a weak cluster point $\nu$ and that there exists a subnet $(\mathcal{I},\le)$ such that 
$\nu_{d+1,\tau} \rightarrow \nu$ weakly along the subnet $\tau \in \mathcal{I}$. Thus, $\nu_{d+1,\tau}^{\pi_J} \rightarrow \nu^{\pi_J}$ weakly along the subnet. 
Since $J$ is countable, the subnet can be  chosen as a subsequence $(n_k)_{k \in \mathbb{N}}$ and we get $\nu_{d+1,n_k}^{\pi_J} \rightarrow \nu^{\pi_J}$ weakly. 
Consider two weak cluster points $\nu_1$, $\nu_2$ of $\nu_{d+1,n_k}$ along our specified subsequence $(n_k)_{k \in \mathbb{N}}$. 
Then, $\nu_1^{\pi_J}$, $\nu_2^{\pi_J}$ are weak cluster points of $\nu_{d+1,n_k}^{\pi_J}$. Using the fact that $\nu_{d+1,n_k}^{\pi_J}$ converges, we obtain $\nu_1^{\pi_J}=\nu_2^{\pi_J}$. 
Since all weak accumulation points are uniquely determined by its $J$-dimensional marginal distribution, this implies $\nu_1=\nu_2$. In consequence, $\nu_{d+1,n_k}$ has a unique weak cluster point and thus converges.
\end{proof}

The convergence of financial experiments can be used to obtain a limit of option prices which turns out to be an option price for a suitable payoff in the limit experiment.
Again the ideas are borrowed from statistics. It is quite obvious that whenever statistical experiments are weakly convergent the appertaining Neyman Pearson tests and their power functions converge to the corresponding objects 
of the limit experiment. These results apply to the asymptotics of European put and call options, see Example \ref{calloption2} below.
In general each sequence of tests has a so called limit test given by the limit model.\\
For the sequence of financial experiments $E_n$ of Theorem \ref{konvergenz1}, consider price processes 
$S_{t,n}^i, 1 \le i \le d$, with initial values $S_{0,n}^i=s_{0,n}^i$, such that
$$\dfrac{dQ_{i,n|\mathcal{F}_{t,n}}}{dQ_{n|\mathcal{F}_{t,n}}}=\frac{S_{t,n}^i}{S_t^0 s_{0,n}^i}$$ 
holds where the bond $S_t^0=\exp\left(\int_0^t \rho(u) du\right)$ is kept fixed.
Suppose that the initial values converge, i.e. $s_{0,n}^i \rightarrow s_0^i$ for each $1 \le i \le d$ as $n \rightarrow \infty$.
Let $H_n:\Omega_n \rightarrow \mathbb{R}$ be payoff functions of type \eqref{a} with coefficients $a_{i,n},K_n$ and $a_{i,n} \rightarrow a_i$, $K_n \rightarrow K$ as $n \rightarrow \infty$ for all $1 \le i \le d$,
\begin{equation}\label{hn}
 H_n=\left(\sum_{i=1}^d a_{i,n} S_{T,n}^i - K_n\right) \phi^{(n)}\left(  \tilde{S}_n\right)
\end{equation}
with $\tilde{S}_n:=\left(\left( \frac{S_{t,n}^i}{s_{0,n}^i \exp\left( \int_0^t \rho(u)du\right) }\right)_{t \le T} \right)_{i=1,...d}$
where $\phi^{(n)}:([0,\infty)^d)^I \rightarrow [0,1]$ are suitable tests.

\begin{theorem}\label{limitp}
 Under the assumptions of Theorem \ref{konvergenz1} consider a weak accumulation point $E=(([0,\infty)^d)^I,(\mathcal{B}([0,\infty)^d))^I,\{\nu_1,...,\nu_d,\nu\})$ of $E_n$. 
Then, there exist a subsequence $(n_k)_{n \in \mathbb{N}}$, a test $\phi:([0,\infty)^d)^I \rightarrow [0,1]$ and a payoff function $H$ for the limit experiment $E$ given by 
$$ H=\left(\sum_{i=1}^d a_i S_T^i - K\right)\phi\left( \tilde{S} \right) $$ where
$$S_t^i=s_0^i \exp\left( \int_0^t \rho(u)du\right) \pi_t^i = s_0^i \exp\left( \int_0^t \rho(u)du\right) \dfrac{d\nu_{i|\mathcal{G}_t}}{d\nu_{|\mathcal{G}_t}}$$
such that the price converges
$p_{Q_{n_k}}(H_{n_k}) \rightarrow p_{\nu}(H)$ as $k \rightarrow \infty$.
For the limit we obtain the option price \begin{eqnarray*}
p_{\nu}(H) = \sum_{i=1}^d a_i s_0^i E_{\nu_i}\left( \phi\left( \tilde{\pi} \right)\right)  
-\exp\left( -\int_0^T \rho(u)du\right) K E_{\nu}\left( \phi\left( \tilde{\pi}\right)\right) 
\end{eqnarray*}
given by the limit experiment with $\tilde{\pi}:=\left(\left( \pi_t^i\right)_{t \le T}\right)_{i=1,...,d}$. Again the results also hold for linear combinations \eqref{a2}.
\end{theorem}

The proof relies on a slight extension of LeCam's ``main theorem of asymptotic testing'' and as such it is modified to reduce subnets to subsequences. 
Roughly speaking, LeCam's result states that, along subsequences, power functions of tests converge to the power function of a limiting test for the limit experiment.
Again this type of result is known also for general decision functions. For convenience we give a self contained proof of financial experiments.

\begin{theorem}\label{konvergenz2}
 Under the assumptions of Theorem \ref{konvergenz1} consider a weak cluster point $\{\nu_1,...,\nu_d,\nu\}$ of the sequence of financial experiments on $([0,\infty)^d)^I$ given by the canonical price process
$$\pi^i_t=\dfrac{d\nu_{i|\mathcal{G}_t}}{d\nu_{|\mathcal{G}_t}}, \quad i=1,...,d.$$
Let $\phi_n:\Omega_n \rightarrow [0,1]$ be a sequence of tests. Then, there exists a test $\phi:([0,\infty)^d)^I \rightarrow [0,1]$, called limit test, and  a subsequence $(n_k)_{k \in \mathbb{N}} \subset \mathbb{N}$ with
$$E_{Q_{i,n_k}}(\phi_{n_k}) \rightarrow E_{\nu_i}(\phi) \quad \textnormal{as } k \rightarrow \infty$$
for each $i=1,...,d+1$.
\end{theorem}

\begin{proof}
 This proof is a modification of the proof of LeCam's main theorem of asymptotic testing, see Strasser \cite{Strasser:85a} (sect. 62), Strasser \cite{Strasser:85b} and R\"uschendorf \cite{Rueschendorf:88} (p. 157) for a special case.
We will indicate the steps which are similar to the proof of Theorem \ref{konvergenz1}.\\
By taking subsequences we may assume without restrictions that $E_{Q_{i,n}}(\phi_n) \rightarrow b_i$ as $n \rightarrow \infty$ simultaneously for all $i \le d$. Our aim is to find a test $\phi$ with power $b_i=E_{\nu_i}(\phi)$.
Consider first a weak cluster point $\rho_{d+1}$ of $\mathcal{L}((\phi_n,Y_{T,n})|Q_n)$ on $[0,1]\times ([0,\infty]^d)^I$ where $\rho_{d+1}$ has the second marginal $\nu_{d+1}$ of Theorem \ref{konvergenz1}. 
As in that proof, we have 
$$\mathcal{L}((\phi_{\tau},Y_{T,\tau})|Q_{r,\tau}) \rightarrow \rho_r$$
weakly along a subnet where
\begin{equation}\label{z1}
\dfrac{d\rho_r}{d\rho_{d+1}}=\tilde{\pi}^r_T \ind_{\{\tilde{\pi}^r_T<\infty\}} 
\end{equation}
and $\tilde{\pi}^r_T: [0,1]\times ([0,\infty]^d)^I \rightarrow [0,\infty]$ is the projection on the $r$-th component at time $T$, similar to \eqref{projektion}. With $p_1:[0,1]\times ([0,\infty]^d)^I \rightarrow [0,1]$ we denote the projection on the first component and with $p_2:[0,1]\times ([0,\infty]^d)^I \rightarrow ([0,\infty]^d)^I$ the projection on the second component. By \eqref{z1} the projection $p_2$ is sufficient for $\{\rho_1,...,\rho_{d+1}\}$ with
$$\mathcal{L}(p_2|\rho_r)=\mathcal{L}(S|\nu_r)=:\tilde{\nu}_r$$
where $S:([0,\infty)^d)^I \rightarrow ([0,\infty]^d)^I$ is the embedding and $\nu_1,...,\nu_r$ are the same as in Theorem \ref{konvergenz1}.\\
Now choose $\tilde{\phi}((x_t^i)_{(i,t) \in \{1,...,d\}\times I})=E_{\centerdot}\left(p_1|p_2=(x_t^i)_{(i,t) \in \{1,...,d\}\times I}\right)$
as the conditional expectation which is independent of $\{\tilde{\nu}_1,...,\tilde{\nu}_{d+1}\}$.\\
There exists a subsequence $(n_k)_{k \in \mathbb{N}}$ of the subnet such that the real numbers are convergent
\begin{eqnarray*}
 E_{Q_{r,n_k}}(\phi_{n_k}) &\rightarrow& \int p_1 d\rho_r\\
&=& \int E_{\centerdot}\left(p_1|p_2=(x_t^i)_{(i,t) \in \{1,...,d\}\times I}\right) d\mathcal{L}(p_2|\rho_r) = \int \tilde{\phi} d\tilde{\nu}_r
\end{eqnarray*}
for each $r=1,...,d+1$.\\
As in the proof of Theorem \ref{konvergenz1}, we observe that $([0,\infty)^d)^I$ has $\tilde{\nu}_r$ outer measure $1$. Thus, we may choose $\phi$ as the restriction of $\tilde{\phi}$ on $([0,\infty)^d)^I$ and we get $\int \tilde{\phi} d\tilde{\nu}_r=\int \phi d\nu_r$
where on $([0,\infty)^d)^I$ the measure $\nu_r$ coincides with the outer measure extension of $\tilde{\nu}_r$.
\end{proof}

\hspace{-.63cm}
\textbf{Proof of Theorem \ref{limitp}:}
We observe that, by Theorem \ref{preis}, $p_{Q_n}(H_n)$ is a finite sum of power functions 
$E_{Q_{i,n}}\left(\phi^{(n)}\left(\left(\dfrac{dQ_{1,n|\mathcal{F}_t}}{dQ_{n|\mathcal{F}_t}}\right)_{t \le T},...,\left(\dfrac{dQ_{d,n|\mathcal{F}_t}}{dQ_{n|\mathcal{F}_t}}\right)_{t \le T}\right)\right)$.
By using Theorem \ref{konvergenz2} we find a subsequence $(n_k)_{k \in \mathbb{N}}$ and a test $\phi$ such that
$$E_{Q_{i,n_k}}(\phi^{(n_k)}) \rightarrow E_{\nu_i}(\phi)$$
as $k \rightarrow \infty$ for each $i \le d$.
Thus, Theorem \ref{preis} implies the result.
\hspace*{\fill}\begin{small}$\square$\end{small}
\\
\bigskip
\\
An easy application can be given for the European call option, see Example \ref{calloption} and Remark \ref{np} below.

\begin{example}[Example \ref{calloption} continued]\label{calloption2}
 Let $H_C^n=(S_{T,n}^1-K)\ind_{\{S_{T,n}^1>K\}}$ denote a sequence of European call options, see \eqref{call}. It is well known that the possible limit tests of the Neyman Pearson tests 
$\phi^n\left( \dfrac{dQ_{1,n}}{dQ_n}\right)=\ind_{\left\{\frac{dQ_{1,n}}{dQ_n}>K (s_{0,n}^1)^{-1} \exp\left(-\int_0^T\rho(u) du \right) \right\}}$
are the Neyman Pearson tests $\phi=\ind_{\left\{\frac{d\nu_1}{d\nu}>K (s_0^1)^{-1} \exp\left(-\int_0^T\rho(u) du \right) \right\}}$ of the limit experiment. We conclude that, along subsequences, the option prices converge to the option price of the European call in the limit. It is also well known that the appertaining Bayes risks are convergent. For convenience we assumed continuous likelihood distributions in the limit. Thus, randomization of tests was not necessary. Taking a proper randomization into account, LeCam's theory offers the same result without continuity restrictions.
\end{example}

\section{Discrete time model approximation}\label{section6}
In this section we will derive a Black-Scholes type model as a limit of discrete financial experiments. As application we obtain the convergence of option prices for discrete models to the related Black-Scholes option prices. 
The structure of this section is the following. The ideas are first explained for the classical Cox-Ross-Rubinstein model in Examples \ref{crr2}, \ref{crr3}, \ref{bs}. 
Under regularity conditions $\textit{(B1)}$ and $\textit{(B2)}$ a general limit theorem with It\^{o} type limit models is obtained which implies asymptotic option price formulas.\\
As a motivation we will deal with the case of independent returns of the discounted price process. 
Consider one discounted asset ($d=1$) and the associated financial experiment $\{Q_{1,N},Q_N\}$ where prices are observed at discrete times 
$$I_N=\left\lbrace m_j:=\frac{jT}N : 0 \le j \le N, j \in \mathbb{N}\right\rbrace \subset [0,T]. $$
We get the following representation
\begin{equation}\label{returns}
 \dfrac{dQ_{1,N|\mathcal{F}_{m_k,N}}}{dQ_{N|\mathcal{F}_{m_k,N}}}
=\frac{X^1_{m_k,N}}{X^1_{0,N}}
=\prod_{j=1}^k \frac{X^1_{m_j,N}}{X^1_{m_{j-1},N}}=:\prod_{j=1}^k Z_{j,N}
\end{equation}
where 
$$Z_{k,N}-1=\frac{\prod_{j=1}^k Z_{j,N}-\prod_{j=1}^{k-1} Z_{j,N}}{\prod_{j=1}^{k-1}Z_{j,N}}=\frac{X^1_{m_k,N}-X^1_{m_{k-1},N}}{X^1_{m_{k-1},N}}$$
are the returns of the asset. Observe that
$$\dfrac{dQ_{1,N|\mathcal{F}_{m_k,N}}}{dQ_{N|\mathcal{F}_{m_k,N}}}
=\dfrac{dQ_{1,N|\mathcal{F}_{m_{k-1},N}}}{dQ_{N|\mathcal{F}_{m_{k-1},N}}} \cdot Z_{k,N}$$
holds and therefore $Z_{k,N}=\dfrac{dQ_{1,N}(k)}{dQ_N(k)}$ is itself a likelihood ratio in the experiment 
$\{Q_{1,N}(k),Q_N(k)\}$ complementary to 
$\left\lbrace Q_{1,N|\mathcal{F}_{m_{k-1},N}},Q_{N|\mathcal{F}_{m_{k-1},N}}\right\rbrace $ with respect to 
$\left\lbrace Q_{1,N|\mathcal{F}_{m_k,N}},Q_{N|\mathcal{F}_{m_k,N}}\right\rbrace $.
With increasing $N$ the time grid induced by $I_N$ becomes finer and, in consequence, the returns $Z_{k,N}-1$ will be small in distribution for large $N$. 
At this stage, LeCam's theory offers the concept of $L_1$- and $L_2$-differentiability which now turns out to be a concept for the local perturbation of returns in the financial markets. 
As mentioned before we will assume here that the returns of the assets are independent under $Q_N$. 
Taking Remark \ref{aequi}(b) into account we may assume without restrictions that the financial experiment is the product experiment of complementary experiments
$$Q_{1,N}=\bigotimes_{j=1}^N Q_{1,N}(j), \quad Q_N=\bigotimes_{j=1}^N Q_N(j) \quad \textnormal{as well as }~~
\dfrac{dQ_{1,N}}{dQ_N}=\prod_{j=1}^N Z_{j,N}.$$
The independence assumption proves that all other forms of the underlying financial experiment are equivalent to that product experiment.
To explain the methodology we will investigate a one step model for the first period above. The ideas will then be used to treat time periods in the $N$ step model. To this end, consider a risky asset with starting price $S^1_0$ and price after the first step $S^1_1$ as well as a bond $S^0_0=1$, $S^0_1=1+\rho$. This gives $X^1_1=\frac{S^1_1}{1+\rho}$ and $X^1_0=S^1_0$ as prices for the discounted asset.
We now turn our attention to the returns of the asset. Assume that there exists a probability measure $P_0$ with 
$\int \frac{S^1_1}{S^1_0} dP_0=1$ such that $\frac{S^1_1}{S^1_0}$ is $P_0$ square integrable, which, at least for a small time period, is a reasonable assumption.
Our approach is to decompose the returns in terms of a function $g$ and some $\sigma \ge 0$ by
\begin{equation}\label{tangente}
\frac{S^1_1}{S^1_0} -1 = \sigma g \quad \textnormal{with } \int g dP_0 = 0, \int g^2 dP_0 = 1.
\end{equation}
Here, $\sigma$ can be viewed as a volatility parameter. As the time grid becomes finer in the $N$ step model, $\sigma$ may decrease whereas the shape function $g$ remains the same in our model.

\begin{example}[Example \ref{crr} Cox-Ross-Rubinstein continued]\label{crr2}
 For $\omega \in \{0,1\}$ the one step Cox-Ross-Rubinstein model is given by $S_1^1(\omega)=S_0^1 u^{\omega} d^{1-\omega}$ and $S_1^0=1+\rho$. If we choose positive parameters $a,b$ satisfying $u=1+\frac {\sigma a}{\sqrt{ab}}$ and 
$d=1-\frac {\sigma b}{\sqrt{ab}}$ with $\sigma>0$ small enough, we arrive at the Bernoulli distribution $P_0=B(1,\frac b{a+b})$ and 
$g:\{0,1\} \rightarrow \mathbb{R}$, $g(1)=\frac a{\sqrt{ab}}$, $g(0)=\frac {-b}{\sqrt{ab}}$, which gives us the decomposition \eqref{tangente}.
\end{example}
In general we start below with a given shape function $g$ of the form \eqref{tangente} and $\sigma>0$.
Under regularity conditions, see Lemma \ref{einperiode}, $1+\vartheta g$ is non negative for small $\vartheta \ge 0$. We may introduce a path of probability measures $\vartheta \mapsto P_{\vartheta}$ by
\begin{equation}\label{pfad1}
 \dfrac{dP_{\vartheta}}{dP_0}:=1+\vartheta g \quad \textnormal{for small }\vartheta \ge 0.
\end{equation}
Thus, in terms of statistical experiments, $g$ can be regarded as a tangent attached at $P_0$.\\
We remark that, when the interest rate $\rho$ is zero, $P_0$ is a martingale measure for the special one step model \eqref{tangente}.
The general approach \eqref{pfad1} leads to the following lemma.

\begin{lemma}\label{einperiode}
 Suppose that $\rho >0$ holds and that the essential infimum $\textnormal{essinf}_{P_0}(g)> -\frac{\sigma}{\rho}$ of $g$ with respect to $P_0$ is bounded from below. 
Within the family \eqref{pfad1} the choice $\vartheta=\frac{\rho}{\sigma}$ is the unique parameter such that $Q:=P_{\frac {\rho}{\sigma}}$ is a martingale measure for the one step discounted price model, i.e.
\begin{equation}
\frac{X_1^1}{X_0^1}=\frac{1+\sigma g}{1+\rho}=\dfrac{dQ_1}{dQ}
\end{equation}
where $Q_1$ is defined by that equation.
\end{lemma}

\begin{proof}
 We first observe that $1+\frac{\rho}{\sigma} g$ is a positive $P_0$ density.
 Using \eqref{tangente} and \eqref{pfad1} we see that 
$$\int (1+ \sigma g) dP_{\frac {\rho}{\sigma}}
=\int (1+\sigma g) dP_0 + \int \left( {\frac {\rho}{\sigma}}g + \rho g^2\right)  dP_0 = 1 + \rho$$
and hence $$\int \frac{X_1^1}{X_0^1} dP_{\frac {\rho}{\sigma}}= \int \frac 1{1+\rho} \frac{S_1^1}{S_0^1}dP_{\frac {\rho}{\sigma}}=\int \frac{1+ \sigma g}{1+\rho} dP_{\frac {\rho}{\sigma}} =1,$$ which implies the result.
\end{proof}

The distribution $P_0$ above can be interpreted as a possible real world measure.

\begin{example}[Example \ref{crr2} Cox-Ross-Rubinstein continued]\label{crr3}
${}$\\In the
Cox-Ross-Rubinstein model the probability measures $P_{\sigma}$ and $P_{\frac{\rho}{\sigma}}$ take the simple form of Bernoulli distributions. 
From \eqref{pfad1} we get $\frac{dP_{\sigma}}{dP_0}(\omega)=u^{\omega} d^{1-\omega}$. Since $P_0=B(1,\frac b{a+b})$, this leads to
 $P_{\sigma}=(1-p_{\sigma})\varepsilon_0+p_{\sigma} \varepsilon_1$ with $p_{\sigma}=u \frac{b}{a+b}=\frac{b+\sigma \sqrt{ab}}{a+b}$.
Consequently, the martingale measure $P_{\frac{\rho}{\sigma}}$ is given by $P_{\frac{\rho}{\sigma}}=(1-\tau)\varepsilon_0+\tau \varepsilon_1$ with $\tau=\frac{\sigma b+\rho \sqrt{ab}}{\sigma(a+b)}$, which is precisely the measure $Q$ in Example \ref{crr} obtained by calculating $\tau=\frac{1-\tilde{d}}{\tilde{u}-\tilde{d}}$.
\end{example}

In a slightly more general context than \eqref{tangente} with heterogeneous volatilities we will establish an approximation of It\^{o} type continuous time financial models by discrete time price models with independent returns. 
To this end we use the famous central limit theorem for statistical experiments given by the local asymptotic normality (LAN) of LeCam and H\'{a}jek. 
For this reason we will now consider a general path \eqref{pfad1} $\vartheta \mapsto P_{\vartheta}$ of distributions on $(\Omega,\mathcal{A})$ (extending our motivation above), which turns out to be $L_2$-differentiable with tangent (score function) 
$g$ at $\vartheta=0$ as $\vartheta \downarrow 0$ under the regularity assumptions $\textit{(B2)}$ below, see Strasser \cite{Strasser:85a} for more information and the proof of Theorem \ref{lan}. \\
Roughly speaking, the idea of LeCam's $L_2$-differentiability is as follows. The path $\vartheta \mapsto P_{\vartheta}$ is embedded in the Hilbert space $L_2(P_0)$ by the root of the likelihood ratio
$$\vartheta \mapsto \left( \dfrac{dP_{\vartheta}}{dP_0}\right)^{1/2} \in L_2(P_0)$$
and $g \in L_2(P_0)$ is its Hilbert space derivative of $2\left( \dfrac{dP_{\vartheta}}{dP_0}\right)^{1/2}$ at $\vartheta=0$ with respect to $L_2(P_0)$-norm. 
Recall that under $L_2$-differentiability the central limit theorem for statistical experiments is well known from LeCam's theory. 
Therefore, we obtain a Gaussian experiment for the limit experiment, which is a filtered financial experiment in our framework. Motivated by Lemma \ref{einperiode} we will now introduce the asymptotics of discrete time models. 
This is a generalization of related results of F\"ollmer and Schied \cite{Foellmer and Schied:04}, p. 246.\\
We require the following assumptions, see \eqref{bedingung1}-\eqref{bedingung3}.
We will restrict ourselves to deterministic interest and volatility parameters $\rho_{i,N}$, $\sigma_{i,N}$. An extension to a more general setting is work under progress and would overload this paper.
\smallskip

$\textit{(B1)}$ Suppose that at stage $N$ the bond is given by $S^0_{0,N}=1$ and 
$$S^0_{\frac{kT}N,N}=\prod_{i=1}^k \left( 1+\frac{\rho_{i,N}T}N\right) $$
with interest rates $\frac{\rho_{i,N}T}N \ge 0$ on $\left( \frac{(i-1)T}N,\frac{iT}N\right]$ satisfying the boundary condition\\ $\textnormal{max}\{\rho_{i,N}: i \le N, N \in \mathbb{N}\} \leq R < \infty$.\\
Assume further that $r:[0,T] \rightarrow [0,\infty)$ is a square integrable function such that
\begin{equation}\label{bedingung1}
\prod_{i=1}^{[\frac{tN}T]}\left(1+\frac{\rho_{i,N}T}N \right) \longrightarrow \exp\left( \int_0^t r(u)du\right) \quad \textnormal{as } N \rightarrow \infty 
\end{equation}
holds for each $t \in [0,T]$.
\smallskip

$\textit{(B2)}$ Suppose that $g \in L_2(P_0)$ is a tangent with $\int g dP_0=0$ and $\int g^2 dP_0=1$. 
We require that $g >-C$ holds for some positive constant $C$ and we consider the $L_2$-differentiable path $(P_{\vartheta})_{0 \le \vartheta \le C^{-1}}$ given by \eqref{pfad1} with tangent $g$ at $\vartheta=0$. 
Within the $i$-th period let $\frac{\sigma_{i,N}\sqrt T}{\sqrt N}$ with $0 < \delta \le \sigma_{i,N} \le K$ denote the volatility and let $S_{0,N}^1=1$. The increments of the returns of the asset are now taken from the path $(P_{\vartheta})_{0 \le \vartheta \le C^{-1}}$, i.e. 
at stage $j \le N$ we take $P_N(j):=P_{\frac{\sigma_{j,N}\sqrt T}{\sqrt N}}$. The consideration in \eqref{returns} motivates the model
\begin{equation}\label{preis2}
S^1_{\frac{kT}N,N}(x_1,...,x_k)=\dfrac{d\otimes_{i=1}^k P_N(i)}{dP_0^k}(x_1,...,x_k)=\prod_{i=1}^k \dfrac{dP_N(i)}{dP_0}(x_i).
\end{equation}
Introduce a step function $\sigma_N:[0,T] \rightarrow \mathbb{R}$ by setting 
 $$\sigma_N(u):=\sum_{k=1}^N \sigma_{k,N} \ind_{\left(\frac{(k-1)T}N,\frac{kT}N\right]}(u). $$
Assume that there exists a square integrable function $\sigma:[0,T] \rightarrow \mathbb{R}$ such that
\begin{equation}\label{bedingung3}
 \int_0^T (\sigma_N(u)-\sigma(u))^2 du \longrightarrow 0 \quad \textnormal{as } N \rightarrow \infty.
\end{equation}
Introduce also $\theta_{k,N}\ =\frac{\rho_{k,N}}{\sigma_{k,N}}$.
$${}$$
Under assumptions $\textit{(B1)}$ and $\textit{(B2)}$, the discounted price process can be written as
$$X^1_{\frac{kT}N,N}(x_1,...,x_k)=\prod_{i=1}^k \left( \dfrac{dP_N(i)}{dP_0}(x_i)\Big/ \left( 1+\frac{\rho_{i,N}T}N\right) \right) .$$
Whenever $N$ is large enough, Lemma \ref{einperiode} establishes a martingale measure $Q_N(j)$ for each time period $j$. 
More precisely, within the $j$-th period with parameters $\frac{\sigma_{j,N}\sqrt T}{\sqrt{N}}$ for volatility and $\frac{\rho_{j,N}T}N$ for interest rate, the martingale measure is given by 
$Q_N(j)=P_{\frac{\theta_{j,N}\sqrt T}{\sqrt{N}}}$ since 
$\frac{\rho_{j,N}T}N \left(\frac{\sigma_{j,N}\sqrt T}{\sqrt{N}}\right)^{-1}=\frac{\theta_{j,N}\sqrt T}{\sqrt{N}}$. 
The associated measure $Q_{1,N}(j)$ is defined by
\begin{equation}\label{mm}
  \dfrac{dQ_{1,N}(j)}{dQ_N(j)}(x_j)=\dfrac{dP_N(j)}{dP_0}(x_j)\Big/ \left( 1+\frac{\rho_{j,N}T}N\right).
\end{equation}
Note that, by setting 
$Q_{1,N}=\bigotimes_{j=1}^N Q_{1,N}(j)$ and $Q_N=\bigotimes_{j=1}^N Q_N(j)$, we recover the measures obtained from complementary experiments mentioned in the beginning of section \ref{section6}.
Therefore, $(Q_{1,N}(j),Q_N(j))$ are the independent increments of a filtered financial experiment at stage $N$, i.e.
$$X^1_{\frac{kT}N,N}(x_1,...,x_k)=\dfrac{d\bigotimes_{j=1}^k Q_{1,N}(j)}{d\bigotimes_{j=1}^k Q_N(j)}(x_1,...,x_k).$$

\begin{theorem}\label{lan}
 Assume that $\textit{(B1)}$ and $\textit{(B2)}$ are satisfied and fix $I=[0,T]$. Consider a Black-Scholes type model with volatility $\sigma(.)$, interest rate $r(.)$ and bond price process $S_0(t)=\exp(\int_0^t r(s) ds)$ given by $(Q_{1,t},Q_t)_{t \in I}$ with
\begin{eqnarray*}
\frac{dQ_{1,t}}{dQ_t}&=&\exp\left(\int_0^t \sigma(s) d\bar{W}_s - \int_0^t \frac{\sigma(s)^2}2 ds\right) \\
&=& \exp\left(\int_0^t \sigma(s) dW_s - \int_0^t \left(r(s)+\frac{\sigma(s)^2}2\right)ds\right)
\end{eqnarray*}
where $\bar{W}$ and $\theta(.)=\frac{r(.)}{\sigma(.)}$ are defined as in Example \ref{ito} for dimension $d=1$.
Then, the finite dimensional marginal distributions of the discounted price process 
$\left( X^1_{[\frac t T N]\frac T N,N}\right)_{t \in I} $ converge weakly to the finite dimensional marginals of 
$\mathcal{L}\left( \left( \dfrac{dQ_{1,t}}{dQ_t}\right)_{t \in I}  \Big| Q_T \right)$ under the martingale measure
$Q_N=\bigotimes_{j=1}^N Q_N$ and to $\mathcal{L}\left( \left( \dfrac{dQ_{1,t}}{dQ_t}\right)_{t \in I}  \Big| Q_{1,T} \right)$ under 
the measure $Q_{1,N}=\bigotimes_{j=1}^N Q_{1,N}(j)$, respectively.
\end{theorem}

\begin{remark}
 Theorem \ref{lan} introduces a central limit theorem for financial experiments and price processes. 
Since the tangent $g$ is normalized note that neither the special distributions \eqref{pfad1} of the returns nor the shape of $g$ contribute to the limit. 
Observe that only the volatility and interest rate determine the parameters of the It\^{o} type model here.
\end{remark}

\hspace{-.63cm}
\textbf{Proof of Theorem \ref{lan}:}
 The right-sided $L_2$-differentiability of the path \eqref{pfad1} with derivative $g$ at $\vartheta=0$ follows from routine arguments using H\'{a}jek's (1972) criterion, see Strasser \cite{Strasser:85a}, Theorem 77.3, p. 391 or Torgersen \cite{Torgersen:91}, p. 537. Note that, in the setup of \eqref{pfad1}, the Fisher information is given by
$$\vartheta \mapsto \int \frac{g^2}{1+\vartheta g} dP_0$$ and that it is continuous as $\vartheta \downarrow 0$ by the dominated convergence theorem.\\
In a first step we establish local asymptotic normality (LAN) for the asset price process. 
For fixed $t \in I$ we introduce $n(t):=[\frac t T N]$ and $\tilde{t}:=[\frac t T N]\frac T N$ and we apply LeCam's technique to obtain the stochastic expansion
\begin{equation}\label{lan2}
 \log S^1_{\tilde{t},N}= \log \dfrac{d\otimes_{i=1}^{n(t)} P_N(i)}{dP_0^{n(t)}}=Z_{t,N}-\frac 1 2 \frac T N \sum_{i=1}^{n(t)} \sigma_{i,N}^2 +R_{t,N}
\end{equation}
with $Z_{t,N}(x_1,...,x_{n(t)}):=\sum_{i=1}^{n(t)} \frac{\sigma_{i,N}\sqrt T}{\sqrt N} g(x_i)$ and remainder term $R_{t,N}$. Note that assumption $\textit{(B2)}$ implies the Noether condition (see Strasser \cite{Strasser:85a}, Conditions 79.1, p. 402)
\begin{equation}\label{noether}
 \max_{i \le n(t)} \Big|\frac{\sigma_{i,N}\sqrt T}{\sqrt N}\Big| \rightarrow 0 \quad \text{ and } \quad \frac T N \sum_{i=1}^{n(t)} \sigma_{i,N}^2 \rightarrow \int_{0}^{t} \sigma(s)^2 ds.
\end{equation}
This fact together with the $L_2$-differentiability implies LAN for the stochastic expansion \eqref{lan2}, i.e.
$\mathcal{L}(Z_{t,N}|P_0^N) \rightarrow N\left( 0, \int_0^t \sigma(s)^2 ds\right) $ weakly and $R_{t,N} \rightarrow 0$ in $P_0^N$-probability as $N \rightarrow \infty$. This is a consequence of LeCam's second lemma, see Strasser \cite{Strasser:85a}, Theorem 79.2, p. 402 or Witting and M\"uller-Funk \cite{Witting and MuellerFunk:95}, p. 317, Satz 6.130. Hence,
\begin{eqnarray*}
\mathcal{L}(\log S^1_{\tilde{t},N}|P_0^N) \rightarrow &N&\left( - \frac 1 2\int_0^t \sigma(s)^2 ds, \int_0^t \sigma(s)^2 ds\right)\\
&=& \mathcal{L}\left(\int_0^t \sigma(s) dW_s - \int_0^t \frac{\sigma(s)^2}2 ds\right). 
\end{eqnarray*}
Combining this result with assumption $\textit{(B1)}$ now gives us convergence for the discounted prices process. More precisely, we get
\begin{eqnarray*}
&\mathcal{L}& \left( \log \left( \dfrac{d\otimes_{i=1}^{n(t)} P_N(i)}{dP_0^{n(t)}} \Big/ \prod_{i=1}^{n(t)}\left(1+\frac{\rho_{i,N}T}N \right)\right)\Bigg|P_0^N\right) \\
&\rightarrow& \mathcal{L}\left(\int_0^t \sigma(s) dW_s - \int_0^t \left(r(s)+\frac{\sigma(s)^2}2\right)ds \right). 
\end{eqnarray*}
In a second step we aim to achieve convergence of the asset prices process and the discounted price process under the martingale measure $Q_N=\bigotimes_{j=1}^N Q_N(j)$. Introduce
$$\alpha_{t,N}=\sum_{k=1}^{n(t)}\left(\frac{\rho_{k,N}}{\sigma_{k,N}} \sqrt{\frac T N} \right)^2.$$
Assumptions $\textit{(B1)}$ and $\textit{(B2)}$ imply 
$\limsup_{N \rightarrow \infty} \alpha_{t,N} \le \frac{R^2}{\delta^2}T$ for each $t$. 
Without restrictions we may assume that $\alpha_{t,N}$ converges to a limit $\alpha_t$ as $N \rightarrow \infty$. Otherwise, we could turn to convergent subsequences and it turns out below that we obtain the same limit for every convergent subsequence chosen which leads to the desired result. 
As above, we get LAN by replacing $\sigma_{k,N}$ with $\frac{\rho_{k,N}}{\sigma_{k,N}}$, i.e.
$$ \log \dfrac{d\otimes_{i=1}^{n(t)} P_N(i)}{dP_0^{n(t)}}=\tilde{Z}_{t,N}-\frac 1 2 \frac T N \sum_{i=1}^{n(t)} \theta_{i,N}^2 +\tilde{R}_{t,N}$$ with $\tilde{Z}_{t,N}(x_1,...,x_{n(t)}):=\sum_{i=1}^{n(t)} \frac{\theta_{i,N}\sqrt T}{\sqrt N} g(x_i)$ where 
$\mathcal{L}(\tilde{Z}_{t,N}|P_0^N) \rightarrow N( 0,\alpha_t) $ \\
weakly and $\tilde{R}_{t,N} \rightarrow 0$ in $P_0^N$-probability as $N \rightarrow \infty$. 
Therefore, $\mathcal{L}((Z_{t,N},\tilde{Z}_{t,N})|P_0^N)$ is asymptotically bivariate normal distributed with asymptotic covariance
$$\lim_{N \rightarrow \infty} \frac T N \sum_{i=1}^{n(t)} \sigma_{i,N} \theta_{i,N} 
=\lim_{N \rightarrow \infty} \frac T N \sum_{i=1}^{n(t)} \rho_{i,N} = \int_0^t r(s)ds.$$
LeCam's third lemma now implies
 $$\mathcal{L}\left( Z_{t,N}|Q_N\right)  \rightarrow N\left( \int_0^t r(s)ds, \int_0^t \sigma(s)^2 ds\right) $$
weakly as $N \rightarrow \infty$. Thus, \eqref{lan2} yields
$$\mathcal{L}\left( \log S^1_{\tilde{t},N} |Q_N\right)\rightarrow \mathcal{L}\left(\int_0^t \sigma(s) dW_s + \int_0^t \left(r(s)-\frac{\sigma(s)^2}2\right)ds \right)$$
weakly as $N \rightarrow \infty$. Consequently, for the discounted price processes we obtain
$$\mathcal{L}\left( \log X^1_{\tilde{t},N} |Q_N\right)\rightarrow \mathcal{L}\left( \log \dfrac{dQ_{1,t}}{dQ_t} \Big| Q_T \right).$$
In a last step, the application of the Cram\'{e}r-Wold device shows convergence of the processes for a finite number of times $t_1,...,t_m \in I$. This follows by taking advantage of the linear structure of the variables $Z_{t_j,N}$.\\
This establishes the convergence under $Q_N$. Since the limit experiment consists of mutually equivalent distributions the convergence under $Q_{1,N}$ to the desired limit is immediate from LeCam's theory.
\hspace*{\fill}\begin{small}$\square$\end{small}

\begin{remark}
 Theorem \ref{lan} can be extended to the following more general situation. Let $\vartheta \mapsto P_{\vartheta}$ be a $L_2$-differentiable path with tangent $g$ which is not necessarily bounded from below. It can then be shown that there exists a sequence $Q_N$ of martingale measures for the price process \eqref{preis2} such that \eqref{mm} holds. For this sequence of martingale measures Theorem \ref{lan} carries over. Details are omitted here. In this situation the martingale measures $Q_N$ no longer have to be members of the path $\vartheta \mapsto P_{\vartheta}$ and consequently, the nice interpretation given by lemma \ref{einperiode} is lost.
\end{remark}

\begin{example}\label{bs}
 Consider a homogeneous model with constant volatility $\sigma_{N,k}=\sigma$ and interest rate 
$\rho_{k,N}=\frac N T \left( \exp\left( \frac{rT}N\right) -1 \right) $ over time. In this case, the limit model is the classical geometric Brownian motion where $\sigma(u)=\sigma$ and $r(u)=r$ are constants. Thus, the discrete time financial model converges to the familiar Black-Scholes model with constant volatility. We present two cases.\\
$\textit{(a)}$ (Example \ref{crr3} Cox-Ross-Rubinstein continued)\\
Fix positive parameters $a,b$. By Example \ref{crr3} the path is given by binomial distributions 
$$P_{\vartheta}=B\left(1,\frac b{a+b}\left(1+\vartheta \frac a{\sqrt{ab}}\right)\right).$$
For normalized initial values $S_{1,N}^0=S_{0,N}^0=1$ we arrive at
$$S^1_{\frac{kT}N,N}(\omega_1,...,\omega_k)=\prod_{j=1}^k u_N^{\omega_j} d_N^{1-\omega_j}, \quad \omega_j \in \{0,1\}$$
where $u_N=1+\frac{\sigma a \sqrt T}{\sqrt{Nab}}$ and $d_N=1-\frac{\sigma b \sqrt T}{\sqrt{Nab}}$. 
As martingale measure for the discounted process
$\left(\frac{S_{\frac{kT}N,N}^1}{(1+\frac {rT}N)^k}\right)_{k \le N}$ we obtain
$$Q_N=\left(B\left(1, \frac b{a+b}\left(1+\frac{\rho}{\sigma}\frac {a \sqrt T}{\sqrt{Nab}}\right)\right)\right)^N.$$
Within this parametrization the Cox-Ross-Rubinstein model converges to the geometric Brownian motion model.\\
$\textit{(b)}$ The same holds true for a wider class of models. As example we will consider the non-complete model given by trinomial distributions. Fix positive values $a,b,c$ with $a+b+c=1$ and denote with $P_0$ the trinomial distribution with elementary probabilities $a,b$ and $c$. Each tangent $g$ with $\int g dP_0=0$ and $\int g^2 dP_0=1$ now leads to a path \eqref{pfad1} of trinomial distributions $(P_{\vartheta})_{\vartheta}$ for sufficiently small $\vartheta$.
Thus, $$S^1_{\frac{kT}N,N}(\omega_1,...,\omega_k)= \dfrac{dP_{\frac{\sigma \sqrt T}{\sqrt N}}^k}{dP_0^k}(\omega_1,...,\omega_k)$$
introduces a discrete time model based on a trinomial tree. Turning to a trinomial martingale measure given by our path the model again tends to the geometric Brownian motion.

\end{example}

\begin{remark}\label{np}
$\textit{(a)}$ If we set $Q_1:=Q_{1,T}$ and $Q:=Q_T$ in Theorem \ref{lan} we get the representation $\dfrac{dQ_{1|\mathcal{F}_t}}{dQ_{|\mathcal{F}_t}}=\dfrac{dQ_{1,t}}{dQ_t}$ and as a result the prices in the limit model can be written in the form of Example \ref{ito}.\\
$\textit{(b)}$ Under the assumptions of Theorem \ref{lan} and in the setup of Example \ref{bs} with volatility 
$\sigma_{k,N}=\sigma$ and interest rate $\rho_{k,N}=\frac N T \left( \exp\left( \frac{rT}N\right) -1 \right) $ we get results for the convergence of prices for 
some options with payoffs of the form \eqref{a}. It is well known that LAN implies contiguity in both directions.
Theorem \ref{limitp} yields the existence of tests in the limit model such that the limit of discrete time option prices can be written in terms of power functions of those tests where the limit model is an It\^{o} type model by Theorem \ref{lan}. For discrete time option prices which are given by power functions of Neyman-Pearson tests, like it is the case for European call, European put, Straddle option, Strangle option and others, the limit price must be given by power functions of Neyman-Pearson tests as well, since Neyman-Pearson tests always converge to Neyman-Pearson tests, compare with Example \ref{calloption2}.
\end{remark}

The reasoning done in Remark \ref{np} $\textit{(b)}$ leads to the following Corollary.
\begin{corollary}\label{bscall}
 Consider European call options \eqref{call} $H_{C,N}=(S_{T,N}-K)^+$ in the situation of Example \ref{bs}. With $p_{Q_N}(H_{C,N})$ we denote the option price calculated as expectation with respect to the martingale measure $Q_N=\otimes_{j=1}^N Q_N(j)$. Then,
$$p_{Q_N}(H_{C,N}) \rightarrow p_Q(H_C) \quad \text{as } N \rightarrow \infty$$ 
where $p_Q(H_C)=s_0^1 \Phi(-x + \sigma \sqrt T) - \exp(-rT)K \Phi(-x)$ is the Black-Scholes price for the European call with 
$x=\frac 1 {\sigma \sqrt T} \left[ \log \frac K{s_0^1} - rT + \frac{\sigma^2}2 T\right]$.
\end{corollary}

\begin{remark}
$\textit{(a)}$ 
The option prices $p_{Q_N}(H_{C,N})$ at stage $N$ are not unique in general but the limit price is uniquely determined for each approximation given by Theorem \ref{lan}. For instance, in the case of Example \ref{bs} \textit{(b)}, different kinds of trinomial models following our parametrization may be chosen and all lead to the same option price. It is possible, however, that another completely different trinomial parametrization could result in a different limit experiment and consequently in different option prices.\\
$\textit{(b)}$
Sometimes an exact pricing formula for an option with payoff function $H$ is hard to compute. Our approach suggests the following approximation of various complicated option prices by discretization of the underlying model.
 \begin{itemize}
\item Find a sequence of discrete time financial experiments which converges to the underlying model as the limit model.
\item Introduce  a sequence of payoff functions $H_n$ of type \eqref{hn} where the prices can be easily computed or simulated within the discrete time model at stage $n$. Suppose that $H$ is as in Theorem \ref{limitp} and that the tests given by $H_n$ can asymptotically be identified with the related tests of $H$.
\item If the option price of $H$ is unique, given by \eqref{ph} with martingale measure $Q$, then any choice of martingale measures $Q_n$ for the discrete time models approximates the unknown price, i.e. $p_{Q_n}(H_n) \rightarrow p_Q(H)$ as $n \rightarrow \infty$.
\end{itemize}
\end{remark}

\begin{appendix}
\section{Appendix}
 Consider $\Omega=[0,\infty]^I$, which is a compact space with respect to the product topology. Let $\mathcal{B}(\Omega)$ be the Borel $\sigma$-field generated by the open subsets of $\Omega$.

\begin{lemma}\label{radon}
 Let $\mu$ be a probability measure on $(\Omega, \mathcal{B}([0,\infty])^I)$ where $\mathcal{B}([0,\infty])^I$ denotes the product $\sigma$-field. Then, there exists a unique extension of $\mu$ as Radon probability measure on $(\Omega, \mathcal{B}(\Omega))$.
\end{lemma}

\begin{proof}
 In a first step a proof is given for the product of the unit interval $\Omega':=[0,1]^I$. Let $C(\Omega')$ denote the set of real valued continuous functions on $\Omega'$. Then, the product $\sigma$-field coincides with the Baire $\sigma$-field on $\Omega'$, i.e. $\mathcal{B}([0,1])^I=\sigma(C(\Omega'))$. \\
The proof of this assertion relies on the theorem of Stone-Weierstrass (see Rudin \cite{Rudin:91}, p.122). Let $\mathcal{A} \subset C(\Omega')$ denote the algebra generated by the constant functions and the univariate projections from $\Omega'$ in $[0,1]$. Then, $\mathcal{A}$ is dense in $C(\Omega')$ with respect to the topology of uniform convergence. \\
Now $C(\Omega') \ni f \mapsto \int f d\mu$ is a continuous linear form on $C(\Omega')$. The representation theorem of Riesz yields the unique extension of $\mu$ to a Radon measure.\\
Consider the topological isomorphism $\phi:[0,\infty] \rightarrow [0,1]$ given by $\phi(x)=\frac x {1+x}$ for $x < \infty$ and $\phi(\infty)=1$. Using $\phi$ and its inverse it is easy to see that the result carries over to $[0,\infty]^I$.
\end{proof}
${}$\\
\textbf{Acknowledgments.}\\
The authors thank the referees for helpful comments and hints to the literature.
\end{appendix}

\bibliographystyle{plain}

\bigskip
\noindent
\parbox[t]{.48\textwidth}{
Arnold Janssen\\
Heinrich-Heine-Universit\"at D\"usseldorf \\
Universit\"atsstr. 1\\
40225 D\"usseldorf, Germany\\
janssena@math.uni-duesseldorf.de } \hfill
\parbox[t]{.48\textwidth}{
Martin Tietje\\
Heinrich-Heine-Universit\"at D\"usseldorf \\
Universit\"atsstr. 1\\
40225 D\"usseldorf, Germany\\
tietje@math.uni-duesseldorf.de }

\end{document}

%% file: template007.tex
\usepackage{a4,times,amsmath,theorem,latexsym,amssymb}

\advance\hoffset by -0.3cm

\makeatletter
\def\@maketitle{%
  \newpage
  {\vspace*{-11ex}%
  \small\noindent}
  \null
  \vskip 5em%
  \begin{flushleft}%
    {\LARGE \bf\@title \par}
    \vskip 1.5em%
       {\large\bf \lineskip .5em \@author \par}%
    \vskip 1.0em%
  \end{flushleft}%
  \par
  \vskip 1.5em}
\renewcommand\section{\@startsection {section}{1}{\z@}%
                                   {-3.5ex \@plus -1ex \@minus -.2ex}%
                                   {1.5ex \@plus.2ex}%
                                   {\normalfont\Large\bfseries}}
\renewcommand\subsection{\@startsection{subsection}{2}{\z@}%
                                     {-3.0ex\@plus -2ex \@minus -.2ex}%
                                     {1.0ex \@plus .2ex}%
                                     {\normalfont\large\bfseries}}
\renewcommand\subsubsection{\@startsection{subsubsection}{3}{\z@}%
                                     {-2.5ex\@plus -2ex \@minus -.2ex}%
                                     {0.9ex \@plus .2ex}%
                                     {\normalfont\normalsize\bfseries}}
\renewcommand\paragraph{\@startsection{paragraph}{4}{\z@}%
                                     {-2.0ex \@plus1ex \@minus.2ex}%
                                     {-1em}%
                                     {\normalfont\normalsize\bfseries}}
\renewcommand\subparagraph{\@startsection{subparagraph}{5}{\parindent}%
                                       {-1.5ex \@plus1ex \@minus .2ex}%
                                       {-1em}%
                                      {\normalfont\normalsize\bfseries}}
\def\@evenhead{\footnotesize\thepage\hfil\slshape\leftmark}%
\def\@oddhead{\footnotesize{\slshape\rightmark}\hfil\thepage}%
\numberwithin{equation}{section} \numberwithin{figure}{section}
\numberwithin{table}{section}

\renewcommand{\@makecaption}[2]{\begin{quote}
\footnotesize {\bf #1}~#2
\end{quote}}
\makeatother
\newenvironment{summary}{\vskip\baselineskip \noindent\small\bf Summary\\ \rm}%
{\vskip\baselineskip}
\newenvironment{proof}{{\vskip\baselineskip\noindent\textbf{Proof:}}}%
{\hspace*{.1pt}\hspace*{\fill}\BOX\vskip\baselineskip}

\newcommand{\BOX}{\ensuremath\Box}
\newtheorem{theorem}{Theorem }[section]

\newtheorem{corollary}[theorem]{Corollary}
\newtheorem{definition}[theorem]{Definition}
{\theorembodyfont{\rmfamily}}
{\theorembodyfont{\rmfamily}\newtheorem{example}[theorem]{Example}}
{\theorembodyfont{\rmfamily}}
\newtheorem{lemma}[theorem]{Lemma}

{\theorembodyfont{\rmfamily}\newtheorem{remark}[theorem]{Remark}}
{\theorembodyfont{\rmfamily}}